\documentclass{article}
  \usepackage{amsmath,amssymb,amsthm,amsfonts}
  \usepackage[dvipsnames]{xcolor}
  \usepackage{enumitem}
  \usepackage{tabularx}
  \usepackage{hyperref}
  \usepackage{cleveref}
  \usepackage{graphicx}
  \usepackage{subfigure}
  \usepackage{tikz}
  \usepackage{booktabs}
  \usepackage{multirow}
  \usepackage{authblk}
  \usepackage[top=3.5cm,bottom=4cm,left=4cm,right=4cm]{geometry}

  \graphicspath{{images/}}

  \theoremstyle{plain}
  \newtheorem{thm}{Theorem}[section]
  \crefname{thm}{theorem}{theorems}
  \newtheorem{prop}[thm]{Proposition}
  \crefname{prop}{proposition}{propositions}
  \newtheorem{lemma}[thm]{Lemma}
  \crefname{lemma}{lemma}{lemmas}
  \newtheorem{coro}[thm]{Corollary}
  \crefname{coro}{corollary}{corollaries}

  \theoremstyle{definition}
  \newtheorem{defi}[thm]{Definition}
  \crefname{defi}{definition}{definitions}
  \newtheorem{example}[thm]{Example}
  \crefname{example}{example}{examples}
  \newtheorem{question}{Question}
  \crefname{question}{question}{Questions}
  
  \newtheorem{prob}{Problem}
  \crefname{prob}{problem}{problems}

  \theoremstyle{remark}
  \newtheorem{remark}[thm]{Remark}
  \crefname{remark}{remark}{remarks}

  \crefformat{step}{step (#2S#1#3)}
  \crefmultiformat{step}{steps (#2#1#3)}{ and (#2#1#3)}{, (#2#1#3)}{ and (#2#1#3)}
  \crefrangeformat{step}{steps (#3S#1#4)--(#5S#2#6)}

  \DeclareMathOperator{\id}{Id}
  
  \DeclareMathOperator{\ord}{order}
  
  \DeclareMathOperator{\aff}{AGL}
  
  \DeclareMathOperator{\aut}{Aut}
  \DeclareMathOperator{\inn}{Inn}
  \DeclareMathOperator{\psl}{PSL}
  \DeclareMathOperator{\gl}{GL}
  
  \DeclareMathOperator{\irr}{Irr}
  \DeclareMathOperator{\irrnt}{Irr^*}
  \DeclareMathOperator{\syl}{Syl}

  \DeclareMathOperator{\sym}{Sym}
  \DeclareMathOperator{\Z}{\mathbb{Z}}
  \def\bn{\mathbb{N}}
  \def\bz{\mathbb{Z}}
  \def\br{\mathbb{R}}
  \def\bc{\mathbb{C}}
  \def\ff{\mathbb{F}}
  \def\bh{\mathbb{H}}
  \def\bq{\mathbb{Q}}
  \def\qb{\mathcal{Q}_8}
  \def\trp{\mathbf{1}}
  \def\divides{\;\Big|\;}
  \def\ds{\displaystyle}
  
  \newcommand{\abs}[1]{\left\lvert#1\right\rvert}
  \newcommand{\norm}[1]{\left\lVert#1\right\rVert}
  \newcommand{\set}[1]{\left\{#1\right\}}
  \newcommand{\iprod}[2]{\left[{#1},{#2}\right]}
  \newcommand{\trep}[1]{\mathfrak{1}_{#1}}
  \newcommand{\sgnrep}[1]{\mathfrak{1}^-_{#1}}
  \newcommand{\indrep}[2]{{#1^#2}}
  \newcommand{\resrep}[2]{{{#1}_{#2}}}
  \newcommand{\sconj}[2]{{{#1}^{#2}}}
  \newcommand{\gconj}[2]{{{#1}^{({#2})}}}
  \newcommand{\subgind}[2]{\abs{{#1}:{#2}}}
  \newcolumntype{Y}{>{\raggedleft\arraybackslash$}X<{$}}
  \newcolumntype{C}{>{$}c<{$}}
  \newcolumntype{L}{>{$}l<{$}}
  \newcolumntype{R}{>{$}r<{$}}

  \newcommand\blfootnote[1]{%
    \begingroup
    \renewcommand\thefootnote{}\footnote{#1}%
    \addtocounter{footnote}{-1}%
    \endgroup
  }

  \newcommand{\thetitle}{On Some Applications of Group Representation Theory to Algebraic Problems Related to the Congruence Principle for Equivariant Maps}
  
  \date{}

  \title{\thetitle}
  \author[1]{Zalman Balanov\thanks{balanov@utdallas.edu}}
  \author[2]{Mikhail Muzychuk\thanks{misha.muzychuk@gmail.com}}
  \author[1]{Hao-pin Wu\thanks{hxw132130@utdallas.edu}}
  \affil[1]{Department of Mathematical Sciences,
      the University of Texas at Dallas,
      Richardson, Texas, USA 75080-3021}
  \affil[2]{Department of Computer Science and Mathematics,
      Netanya Academic College, Netanya, Israel 4223587}

  \hypersetup{
    bookmarks=true,
    unicode=true,
    pdfauthor={Zalman Balanov, Mikhail Muzychuk and Hao-pin Wu},
    pdftitle={\thetitle},
    pdfsubject={congruence principle},
    pdfkeywords={equivariance}; {congruence};
        {2-transitive groups}; {group representation}
  }

\begin{document}
  \titlepage
  \maketitle

  \begin{abstract}
  Given a finite group $G$ and two unitary $G$-representations
  $V$ and $W$, possible restrictions on Brouwer degrees of
  equivariant maps between representation spheres $S(V)$ and $S(W)$
  are usually expressed in a form  of congruences modulo the
  greatest common divisor  of lengths of orbits in $S(V)$
  (denoted $\alpha(V)$).
  Effective applications of these congruences is limited by
  answers to the following questions:
  (i) under which conditions, is $\alpha(V)>1$?
  and (ii) does there exist an equivariant map with the degree
  easy to calculate?
  In the present paper, we address both questions.
  We show that $\alpha(V)>1$ for each irreducible non-trivial
  $\mathbb{C}[G]$-module if and only if $G$ is solvable.
  For non-solvable groups, we use $2$-transitive actions
  to construct complex representations with non-trivial
  $\alpha$-characteristic.
  Regarding the second question, we suggest a class
  of Norton algebras without 2-nilpotents giving
  rise to equivariant quadratic maps, which
  admit an explicit formula for the Brouwer degree.
  \end{abstract}

  \blfootnote{{\em Acknowledgement.}
      The authors are thankful to A.~Kushkuley for useful discussions.
      The first and third authors acknowledge the support from
      National Science Foundation through grant DMS-1413223.
      The second author acknowledge the support from
      Israeli Ministry of Absorption.}

  \section{Introduction}
  \subsection{Topological motivation}
  The methods based on the usage of Brouwer degree and
  its infinite dimensional generalizations are unavoidable
  in many mathematical areas which, at first glance,
  have nothing in common: (i) qualitative investigation of
  differential and integral equations arising in mathematical physics
  (existence, uniqueness, stability, bifurcation of solutions
  (see \cite{KrasnoselskiiZabreiko,Nirenberg,Deimling})), 
  (ii) combinatorics (equipartition of mass (see \cite{BlagoZig,Z})),
  geometry (harmonic maps between surfaces (see \cite{EelsWood,JostSchoen})),
  to mention a few.
  In short, given a continuous map $\Phi:M\to N$ of manifolds
  of the same dimension, the Brouwer degree $\deg(\Phi)$ is an integer 
  which can be considered as an algebraic count of solutions to equation
  $\Phi(x)=y$ for a given $y\in N$
  (for continuous functions from $\mathbb R$ to $\mathbb R$,
  the Brouwer degree theory can be traced to Bolzano-Cauchy
  Intermediate Value Theorem).

  In general, practical computation of the Brouwer degree is a
  problem of formidable complexity.
  However, if $\Phi$ respects some group symmetries of 
  $M$ and $N$ (expressed in terms of the so-called equivariance,
  see \Cref{subsec:groups_actions}), then the computation of
  $\deg(\Phi)$ lies in the interplay between topology and group 
  representation theory.
  Essentially, symmetries lead to restrictions on possible
  values of the degree.  These restrictions
  (typically formulated in a form of {\it congruencies})
  have been studied by many authors using different techniques
  (see, for example, \cite{Kras1955,Kushkuley-Balanov,
  Ize-Vignoli,Borisovich-Fomenko,Bartsch,Atiyah-Tall,tomDieck,
  Komiya,Dancer,Marzantowicz}
  and references therein (see also the survey \cite{Steinlein1})).
  The following statement (which is  
  a particular case of the so-called {\em congruence principle}
  established in \cite{Kushkuley-Balanov}, Theorems 2.1 and 3.1)
  is the starting point for our discussion.

  \medskip
  \noindent
  {\bf Congruence principle:}
  {\it Let $M$ be  a compact, connected, oriented,
  smooth $n$-dimensional manifold on which a finite group $G$ acts smoothly, 
  and let $W$ be an orthogonal $(n+1)$-dimensional
  $G$-representation. Denote by $\alpha(M)$
  the greatest common divisor of lengths of $G$-orbits occurring in $M$.
  Assume that there exists an equivariant map $\Phi:M\to W\setminus\set{0}$.} 
  {\it Then, for {\bf any} equivariant map}
  $\Psi:M\to W\setminus\set{0}$, {\it one has} 
  \begin{equation}\label{eq:congruence-principle-general}
    \deg(\Psi)\equiv\deg(\Phi)\pmod{\alpha(M)}.
  \end{equation}
 
  Clearly, the congruence principle contains a non-trivial
  information only if $\alpha(M)>1$
  (for example, if a non-trivial group $G$ acts freely on $M$,
  then $\alpha(M)=\abs{G}>1$).
  Also, the congruence principle can be effectively applied
  only if there exists a ``canonical" equivariant map
  $\Phi:M\to W\setminus\set{0}$ with $\deg(\Phi)$ easy to calculate
  (for example, if $M$ coincides (as a $G$-space)
  with the unit sphere $S(W)$ in $W$, 
  then one can take $\Phi:=\id$, in which case,
  for {\em any} equivariant map $\Psi:S(W)\to W\setminus\set{0})$,
  one has:
  $\deg(\Psi)\equiv\deg(\id)=1\pmod{\alpha(S(W))}$;
  in particular, $\deg(\Psi)\neq0$ provided $\alpha(S(W))>1$).
  This way, we arrive at the following two problems:

  \begin{prob}\label{prob:A}
    Under which conditions on $M$, is $\alpha(M)$ greater than 1? 
  \end{prob}

  \begin{prob}\label{prob:B}
    Under which conditions on $M$ and $W$,
    does there exist an equivariant map
    $\Phi:M\to W\setminus\set{0}$
    with $\deg(\Phi)$ easy to calculate?
  \end{prob}

  Assume, in addition, that $V$ is an orthogonal
  $G$-representation and $M=S(V)$
  (recall that $S(V)$ is called a $G$-representation sphere). 
  Then: (i) Problem A can be traced to the classical result of
  J.~Wolf \cite{Wolf} on classification of finite groups acting
  freely on a finite-dimensional sphere,
  (ii) both Problems A and B are intimately related to a
  classification of $G$-representations up to a certain
  (non-linear) equivalence (see \cite{Atiyah-Tall,tomDieck,Adams,LW}). 


  A study of numerical properties of orbit lengths of finite linear groups
  has a long history and can be traced back to H.~Zassenhaus \cite{Z}.
  A special attention was paid to studying
  regular orbits, orbits of coprime lengths, etc.,
  in the case of the ground field of positive characteristic
  (see \cite{GLPT} for a comprehensive account about the current
  research in this area).
  To the best of our knowledge,
  the case of zero characteristic was not as well studied
  as the one of positive characteristic.
  It seems that the invariant introduced in our paper
  (the $\alpha$-characteristic of a linear representation)
  has not been studied in detail before.
  
  
  The {\em goal} of this paper is to develop some algebraic
  techniques allowing one to study \Cref{prob:A,prob:B} for finite
  solvable and $2$-transitive groups.  We are focused on the
  situation when  $V$ and $W$ are complex unitary $G$-representations
  of the same dimension and $M=S(V)$ is a $G$-representation sphere
  (in this case, we set $\alpha(V)=\alpha(S(V))$ and call it
  $\alpha$-characteristic of $V$).
  However, some of our results 
  (see \Cref{coro:maps-manifold-solv}) are formulated for
  equivariant maps of $G$-manifolds.

  \subsection{Main results and overview} 
  \paragraph{(A)}
  If $V$ and $U$ are (complex unitary) $G$-representations, then
  $\alpha(U\oplus V)=\gcd\set{\alpha(U),\alpha(V)}$.
  This simple observation suggests to study \Cref{prob:A}
  first for $S(V)$, where $V$ is an irreducible representation.
  By combining the main result from \cite{KaplanLevy} with several
  group theoretical arguments, we obtain the following result:
  $G$ is solvable if and only if $\alpha(V)>1$ for
  {\em any} non-trivial irreducible $G$-representation
  (see \Cref{thm:solvable_groups}).
  Among many known characterizations of the class of
  (finite) solvable groups,
  we would like to refer to Theorem 3.7 from \cite{Bartsch}
  (where a concept of admissible representations is used)
  as the result close in spirit to ours.
  Also, if $G$ is nilpotent, we show that for any non-trivial
  irreducible $G$-representation, there exists an orbit 
  $G(x)$ in $S(V)$ such that $\abs{G/G_x}=\alpha(V)$
  (see \Cref{prop:nilpotent_realize_alpha}). 

  On the other hand, we discovered that a sporadic group
  (the Janko Group $J_1$ (see \cite{Janko}))
  satisfies the following property: {\em all}
  irreducible $J_1$-representations have the $\alpha$-characteristic
  equals $1$ (recall that $J_1$ is of order $175560$
  and admits $15$ irreducible representations).
  With these results in hand, we arrived at the following question:
  Given a (finite) non-solvable group $G$ different from $J_1$,
  does there exist an easy way to point out an
  irreducible $G$-representation $V$ with $\alpha(V)>1$?
  In this paper, we focus on the following setting:
  Given $H<G$, take the $G$-action on $G/H$ by left translations
  and denote by $V$ the augmentation submodule of the associated
  permutation $G$-representation $\mathbb C\oplus V$.
  It turns out that $\alpha(V)>1$ if and only if $|G/H|=q^{k}$,
  where $q$ is a prime (see \Cref{lemma:2trans_groups}).
  Combining this observation with the classification of
  $2$-transitive groups (see \cite{Cameron}, for example)
  allows us to completely describe faithful augmented
  modules $V$ associated with $2$-transitive group
  $G$-actions on $G/H$ such that $\alpha(V)>1$
  (see \Cref{thm:2trans_groups}). 

  Finally, it is possible to show that if $H\trianglelefteq G$,
  $V$ is an $H$-representation and $W$ is a $G$-representation
  induced from $V$, then, $\alpha(V)$ divides $\alpha(W)$.
  This observation suggests the following question:
  under which conditions, does $\alpha(V)=1$ imply $\alpha(W)=1$?
  We answer this question affirmatively
  assuming that $V$ is irreducible and
  $G/H$ is solvable (see \Cref{prop:normal}). \\

  \noindent
  \paragraph{(B)}
  In general, \Cref{prob:B} is a subject of the equivariant
  obstruction theory (see \cite{tomDieck,Bartsch} and references therein)
  and is far away from being settled even in relatively simple cases.
  On the other hand, if $W$ is a subrepresentation of the
  $m$-th symmetric power of $V$, then one can look for a required map
  in the form of a $G$-equivariant {\em $m$-homogeneous} map
  $\Phi:S(V)\to W\setminus\set{0}$,
  in which case $\deg(\Phi)=m^n$.
  In particular case when $m=2$,
  \Cref{prob:B} reduces to the existence of a commutative
  (in general, non-associative) bilinear multiplication
  $\ast:V\times V\to V\subset\mathrm{Sym}^2(V)$ satisfying two properties:
  (i) $\ast$ commutes with the $G$-actions,
  and (ii) the complex algebra $(V,\ast)$ is free from $2$-nilpotents.
  Combining this idea with the techniques related to the
  so-called Norton algebra (see \cite{Cameron-Goethals-Seidel}),
  we establish the existence of an equivariant quadratic map between
  two non-equivalent $(n-1)$-dimensional
  $S_n$-representations (having the same symmetric square)
  taking non-zero vectors to non-zero ones, provided that $n$ is odd
  (see \Cref{thm:sn}). For $n=5$, we give an explicit
  formula of such a map. \\
  
  After the Introduction, the paper is organized as follows.
  \Cref{subsec-prelim} contains preliminaries related to groups
  and their representations.
  In \Cref{sec:alpha-cracteristic},
  we first consider functorial properties of $\alpha$-characteristic
  (see \Cref{prop:alpha_func}). Next, we focus on solvable and
  nilpotent groups and prove \Cref{thm:solvable_groups} and
  \Cref{prop:nilpotent_realize_alpha}. 
  2-transitive actions are considered in \Cref{sec:4},
  while induced representations with trivial
  $\alpha$-characteristic are considered in \Cref{sec:5}.
  \Cref{sec:6} is devoted to the existence of quadratic maps
  relevant to \Cref{prob:B}. In the concluding
  \Cref{sec:7}, we consider applications of the obtained results
  to the congruence principle.    


  \section{Preliminaries}\label{subsec-prelim}
  \subsection{Groups and Their Actions}\label{subsec:groups_actions}
  This subsection collects some basic facts about finite
  groups and their actions that are used in our paper.
  Although the material given here is well-known to any
  group theorist, we decided to include it here,
  because we expect that the paper could be of interest for
  mathematicians working outside the group theory.

  Throughout the paper, we consider only
  finite groups if no otherwise is stated,
  and by $G$, we always mean a finite group.

  For any $G$, denote by $\aut(G)$ (resp.~$\inn(G)$)
  the group of automorphisms
  (resp.~inner automorphisms) of $G$,
  by $e$ the identity of $G$ and
  by $\trp$ the trivial group or
  the trivial subgroup of $G$.

  Given $H,K<G$, set
  $HK:=\set{hk\in G:h\in H,\,k\in K}$.
  Given a prime $p$, denote by
  $\syl_p(G)$ the collection of
  Sylow $p$-subgroups of $G$.
  Recall an important characterization of
  solvable groups from \cite{KaplanLevy}:

  \begin{thm}\label{thm:KaplanLevy}
    Let $p_1,p_2,\cdots,p_k$ be
    a sequence of all distinct prime
    factors of $\abs{G}$.
    Then, $G$ is solvable if and only if
    $G=P_1P_2\cdots P_k$ for {\em any} choice of
    $P_j\in\syl_{p_j}(G)$, $j=1,\dots,k$.
  \end{thm}

  Recall that $N\trianglelefteq G$
  is called a {\em minimal normal subgroup}
  if $N$ is non-trivial and contains
  no other non-trivial normal subgroups of $G$.
  The {\em socle} of $G$
  is the subgroup generated by
  all minimal normal subgroups of $G$.
  The following result is well-known
  (see\cite{Isaacs}).

  \begin{prop}\label{prop:solvable_minimal_normal}
    A minimal normal subgroup of a
    solvable group is elementary abelian.
  \end{prop}

  Let $X$ be a $G$-space.
  For any $x\in X$, denote
  by $G_x$ the {\em isotropy} (stabilizer) of $x$ and by
  $G(x)$ the $G$-orbit of $x$ in $X$.
  We call the conjugacy class of $G_x$ the
  {\em orbit type} of $x$ and denote by
  $\Phi(G;S)$ the collection of orbit
  types of points in $S\subset V$.
  For any $H<G$, denote by
  $X^H:=\set{x\in X:hx=x\mbox{ for all } h\in H}$
  the set of $H$-fixed points in $X$.

  If $\abs{X}\geq2$,
  we say that $G$ acts {\em $2$-transitively} on $X$
  if for any $a,b,c,d\in X$, $a\neq b$, $c\neq d$,
  there exists $g\in G$ such that $ga=c$ and $gb=d$.
  Since any transitive (in particular, $2$-transitive)
  action is equivalent to the $G$-action on the coset space
  $G/H$ by left translation for some $H<G$,
  the existence of $2$-transitive
  $G$-action is actually an intrinsic property of $G$.
  Therefore, we adopt the following definition.

  \begin{defi}\label{def:2trans_group}
    $G$ is called a {\em $2$-transitive group}
    if it admits a faithful $2$-transitive action,
    or equivalently, $G$ acts $2$-transitively
    on $G/H$ (by left translation) for some $H<G$.
  \end{defi}

  Suppose $X$ and $Y$ are (topological) $G$-spaces.
  A continuous map $f:X\to Y$ is called {\em $G$-equivariant}
  if $f(gx)=gf(x)$ for any $g\in G$ and $x\in X$.
  Note that in this case, $f$ takes $H$-fixed points
  in $X$ to $H$-fixed points in $Y$
  (i.e., $f(X^H)\subset Y^H$) for any $H<G$).
  If, in addition, $X$ and $Y$ are {\em linear}
  $G$-spaces, a $G$-equivariant map $f:X\to Y$ is
  called {\em admissible} if $f^{-1}(0)=\set{0}$.
  We refer to \cite{Bredon} and \cite{Kushkuley-Balanov}
  for the equivariant topology background.

  \subsection{Group Representations}\label{subsec:repn}
  Throughout the paper, we consider only
  {\em finite-dimensional complex unitary}
  representations, and by $\rho$ (resp.~$V$ and $\chi$),
  we always mean a $G$-representation (resp.~the
  associated vector space and the affording
  character) if no otherwise is stated.

  Let $K$ be an arbitrary field.
  Denote by $K[G]$ the group algebra of $G$ over $K$.
  For any $\rho$, we will simply denote by the same symbol
  the extension of $\rho$ to $K[G]$
  (i.e., depending on the context,
  it is possible that $\rho:G\to\gl(V)$ or $\rho:K[G]\to\mathrm{End}(V)$).

  For any $G$, denote by $\trep{G}$ the
  {\em trivial representation} or
  {\em trivial character} of $G$ (depending on the context)
  and by $\irr(G)$ (resp.~$\irrnt(G)$) the collection of
  {\em irreducible} (resp.~{\em non-trivial irreducible})
  $G$-representations.

  For any representation $\rho:G\rightarrow GL(V)$,
  denote by $\rho(G)(x)$
  or $G(x)$ the $G$-orbit of $x$ for any $x\in V$.
  In addition, set $\Phi(\rho):=\Phi(G;S(V))$,
  where $S(V)$ stands for the unit sphere in $V$.

  If $\rho$ and $\sigma$ are $G$-representations,
  then $\iprod{\rho}{\sigma}$ will stand for the {\em scalar product}
  of their characters.


  Let $H<G$. For any $G$-representation $\rho$
  with character $\chi$, denote by $\resrep{\rho}{H}$
  and $\resrep{\chi}{H}$ the {\em restriction} of
  $\rho$ and $\chi$ to $H$, respectively.
  On the other hand, for any $H$-representation $\psi$
  with character $\omega$, denote by $\indrep{\psi}{G}$
  and $\indrep{\omega}{G}$ the {\em induced representation}
  and the {\em induced character} of $\psi$ to $G$,
  respectively.

  Let $\sigma$ be an automorphism of $G$.
  Denote by $\sconj{\rho}{\sigma}$ (resp.~$\sconj{\chi}{\sigma}$)
  the composition $\rho\circ\sigma$ (resp.~$\chi\circ\sigma$).
  It is clear that: (i) $\sconj{\rho}{\sigma}$ is a
  $G$-representation affording character $\sconj{\chi}{\sigma}$,
  and (ii) if $\rho$ is irreducible, so is $\rho^\sigma$.
  If, in particular, $\sigma:g\mapsto ugu^{-1}$
  for some $u\in U\trianglerighteq G$,
  denote by $\gconj{\rho}{u}$ (resp.~$\gconj{\chi}{u}$)
  the composition $\rho\circ\sigma$ (resp.~$\chi\circ\sigma$)
  instead of $\sconj{\rho}{\sigma}$ (resp.~$\sconj{\chi}{\sigma}$).
  In such a case, $\gconj{\rho}{u}$ (resp.~$\gconj{\chi}{u}$)
  is said to be {\em $U$-conjugate} to $\rho$ (resp.~$\chi$).

  Recall the following result
  for permutation representations
  associated to $2$-transitive actions
  (see \cite{Isaacs1976}).

  \begin{prop}\label{prop:2trans_perep}
    Let $G$ act transitively on $X$.
    Then, the permutation representation 
    associated to this action
    is equivalent to $\trep{G}\oplus\rho$,
    where all irreducible components of $\rho$
    are non-trivial. If, in addition, $G$
    acts $2$-transitively on $X$,
    then $\rho$ is irreducible.
  \end{prop}

  The $G$-representation $\rho$ in
  \Cref{prop:2trans_perep} will play
  an essential role in our consideration.
  We adopt the following definition.

  \begin{defi}\label{def:aug_subm}
    Following \cite{Griess},
    we call the representation $\rho$ from
    \Cref{prop:2trans_perep}
    the {\em augmentation representation}
    associated to the
    transitive $G$-action on $X$
    (resp.~$G/H$ by left translation)
    and denote it by
    $\rho_{(G;X)}^a$ (resp.~$\rho_{[G;H]}^a$)\footnote{The
    corresponding module will be called the augmentation module.
    We do not use a special notation for it.}.
    In particular, denote by $\varrho_2(G)$
    the collection of all its non-isomorphic augmentation representations
    arised from $2$-transitive actions of $G$.
  \end{defi}

  We refer to \cite{Serre}, \cite{Brocker-tomDieck} and \cite{Isaacs1976}
  for the representation theory background and
  notation frequently used in this paper.

  \section{$\alpha$-characteristic of $G$-representations}
      \label{sec:alpha-cracteristic}
  The following definition is crucial for our discussion.

  \begin{defi}\label{defi:alpha-char}
    For a $G$-representation $\rho:G\to\gl(V)$,
    we call
    \begin{align*}
      \alpha(\rho)=\alpha(G,S(V)):=\;&
          \gcd\set{\abs{G(x)}:x\in S(V)}\\
      =\;&\gcd\set{\abs{G/H}:(H)\in\Phi(\rho)}
    \end{align*}
    the {\em $\alpha$-characteristic} of $\rho$.
    We will call the $\alpha$-characteristic
    of a representation {\em trivial} if it takes value $1$.
  \end{defi}

  \noindent
  Note that $\alpha$-characteristic admits
  the following functorial properties.

  \begin{prop}\label{prop:alpha_func}
    Suppose $\rho$ is a $G$-representation.
    \begin{enumerate}[leftmargin=1.75em]
    \renewcommand{\labelenumi}{\rm(\alph{enumi})}
    \item Let $H<G$.
      Then, $\alpha(\resrep{\rho}{H})$ divides $\alpha(\rho)$.
    \item Let $H\trianglelefteq G$ and $\theta$ be an
      $H$-representation.
      Then, $\alpha(\theta)$ divides $\alpha(\theta^G)$.
    \item Let $\sigma$ be an automorphism of $G$.
      Then, $\alpha(\rho^\sigma)=\alpha(\rho)$.
    \item Let $\ff$ be a splitting field of the
      group algebra $\bq[G]$ and $\sigma$
      an automorphism of $\ff$.
      Then, $\alpha(\rho^\sigma)=\alpha(\rho)$
      for $\rho\in\irr(G)$.
    \item Let $\psi$ be another $G$-representation.
      Then, $\alpha(\rho\oplus\psi)=
      \gcd\set{\alpha(\rho),\alpha(\psi)}$.
    \end{enumerate}
  \end{prop}

  \begin{proof}
    Here we prove part (b) only
    since other properties are quite
    straightforward from \Cref{defi:alpha-char}.
    Denote by $V$ and $W$ the
    representation spaces
    of $\indrep{\theta}{G}$ and $\theta$, respectively.
    Take an arbitrary non-zero $v\in V$.
    It suffices to show that $\alpha(\theta)$ divides $\abs{N(v)}$.
    Since $V$ is induced by $W$, one has
    $v=\sum{g_iw_i}$, where
    $\set{g_i}$ is the complete set of representatives
    of $N$-cosets in $G$ and $w_i\in W$.
    Without loss of generality, assume that $w_1\neq 0$.
    Take $n\in N_v$.
    Since $nv=\sum{g_i(g_i^{-1}ng_i)w_i}$
    and $N$ is normal,
    we conclude that $n\in N_v$ if and only if 
    $n\in g_iN_{w_i}g_i^{-1}$ for every $i$. In particular,
    $n\in g_1N_{w_1}g_1^{-1}$ implying
    $N_v<g_1N_{w_1}g_1^{-1}<N$. Therefore,
    $\abs{N(w_1)}=\subgind{N}{N_{w_1}}$
    divides $\abs{N(v)}=\subgind{N}{N_v}$ and
    the result follows from the fact that
    $\alpha(\theta)$ divides $\abs{N(w_1)}$.
  \end{proof}

  \begin{remark}\label{rmk:alpha_func}
    (i) According to \Cref{prop:alpha_func}(e),
    given a $G$-representation $\rho$ (possibly reducible),
    one can evaluate $\alpha(\rho)$ by
    computing $\alpha$-characteristics
    of all irreducible components of $\rho$.
    \vskip .3em
    
    \noindent
    (ii) The conclusion of \Cref{prop:alpha_func}(b)
    is not true if $H$ is not normal in $G$.
    The simplest example is provided by the group
    $G=S_3$ with $H$ to be an order two subgroup.
    If $\theta$ is a non-trivial reperesentation of $H$,
    then $\theta^G = \rho_1\oplus\rho_2$ where
    $\dim(\rho_1)=1,\dim(\rho_2)=2$. 
    In this case,
    $\alpha(\theta)=2$ while $\alpha(\rho_1)=2$ and $\alpha(\rho_2)=3$,
    so that $\alpha(\theta^G)=1$.
    \vskip .3em

    \noindent
    (iii) One could think that there always exists an irreducible
    constituent $\rho$ of $\theta^G$ with $\alpha(\theta)=\alpha(\rho)$.
    But this is not true as the following example shows.
    Take $G=\qb$, a quaternion group of order eight, and let $H$ be
    its cyclic subgroup of order $4$. If $\theta$ is a faithful
    irreducible representation of $H$, then $\theta^G$ is an irreducible
    $2$-dimensional $G$-representation.
    In this case, $\alpha(\theta)=4$ while $\alpha(\theta^G)=8$.
  \end{remark}

  \begin{example}
    Computation of $\alpha(\rho)$
    for $\rho\in\irrnt(G)$ involves finding
    maximal orbit types $(H)\lneq (G)$ of $\rho$.
    For example, the group $A_5$
    admits four non-trivial irreducible
    representations with
    the lattices of orbit types shown in 
    \Cref{fig:a5_irr_lat_orbtypes}.
    Then, for each $\rho\in\irrnt(A_5)$,
    $\alpha(\rho)$ is the greatest
    common divisor of indices of
    proper subgroups which appear
    in the lattice.
    The result is shown in \Cref{tbl:a5_chartbl}.
  \end{example}

  \begin{remark}
    Note that the character table of a group does not determine
    the $\alpha$-characteristic of its irreducible representations.
    For example, $D_8$ and $Q_8$ have the same character table
    while the $\alpha$-characteristic of their unique
    $2$-dimensional irreducible representations are
    distinct ($4$ for $D_8$ and $8$ for $Q_8$).
  \end{remark}

  \subsection{$\alpha$-characteristic of Solvable Group Representations}
  \Cref{prob:A} together with \Cref{rmk:alpha_func}
  give rise to the following questions.

  \begin{question}\label{q:A}
    Does there exist a non-trivial group $G$ such that
    $\alpha(\rho)=1$ for {\em any} $\rho\in\irr(G)$?
  \end{question}

  \begin{question}\label{q:B}
    Does there exist a reasonable class of groups $\mathcal{A}$
    such that for {\em any} $G\in\mathcal{A}$, one has
    \begin{align}\label{eq:dagger}
      \mbox{$\alpha(\rho)>1$ for {\em any} $\rho\in\irrnt(G)$?}
      \tag{$\dagger$}
    \end{align}
  \end{question}

  \begin{question}\label{q:C}
    Given a group $G$ which is neither in the case of
    \Cref{q:A} nor \Cref{q:B},
    how can one find a
    $G$-representation $\rho$ with $\alpha(\rho)>1$?
  \end{question}

  \noindent
  An affirmative answer to \Cref{q:A} is given by
  the following example.

  \begin{example}
    The Janko Group $J_1$ has 15 irreducible
    representations---all of them admit
    trivial $\alpha$-characteristics.
  \end{example}

  \noindent
  We give a complete answer to \Cref{q:B}
  in the rest of this subsection and
  address \Cref{q:C} in \Cref{sec:4,sec:5}.
  The following example is the
  starting point for our discussion.

  \begin{example}\label{example:basic_A_groups}
    If $G$ is abelian or a $p$-group, then
    \eqref{eq:dagger} is true.
  \end{example}

  \noindent
  We will show that the following statement is true.

  \begin{thm}\label{thm:solvable_groups}
    $G$ is solvable if and only if
    $\alpha(\rho)>1$ for any $\rho\in\irrnt(G)$.
  \end{thm}

  \begin{remark}
    As it will follow from the proof,
    the conclusion of the \Cref{thm:solvable_groups}
    remains true if one replace the complex field by
    an algebraically closed field of a characteristic
    coprime to $|G|$.
  \end{remark}

  Let us first present two lemmas required
  for the proof of necessity in
  \Cref{thm:solvable_groups}.

  \begin{lemma}\label{lemma:solvable_group1}
    Let $\rho\in\irrnt(G)$ and $P\in\syl_p(G)$.
    Then $\alpha(\rho_P)=\alpha(\rho)_p$,
    where $\alpha(\rho)_p$ is the highest $p$-power that divides
    $\alpha(\rho)$. In addition,
    the following statements are equivalent.

    \begin{enumerate}
    \renewcommand{\labelenumi}{(\roman{enumi})}
    \item $p$ divides $\alpha(\rho)$.
    \item $P_x\lneq P$ for any $x\in S(V)$.
    \item $\iprod{\resrep{\chi}{P}}{\trep{P}}=0$.
    \end{enumerate}
  \end{lemma}

  \begin{proof}
    By \Cref{prop:alpha_func}(a), $\alpha(\rho_P)$ divides $\alpha(\rho)$.
    Since $\alpha(\rho_P)$ is a $p$-power, we conclude that
    $\alpha(\rho_P)\,|\,\alpha(\rho)_p$.

    Let us show that $\alpha(\rho)_p$ divides the
    cardinality of every $G$-orbit in $S(V)$.
    Let $O\subseteq S(V)$ be a $G$-orbit.
    By Exercise 1.4.17 in \cite{DM}, the length of every
    $P$-orbit in $O$ is divisible by $\abs{O}_p$. Therefore,
    $\alpha(\rho)_p$ divides length of every $P$-orbit in $O$.
    Hence, $\alpha(\rho)_p$ divides the length of every $P$-orbit in $S(V)$.
    Hence, $\alpha(\rho)_p\mid\alpha(\rho_P)$.

    \begin{description}[leftmargin=0cm]
    \item[\rm(i)$\implies$(ii).]
      Since $\alpha(\rho_P)=\alpha(\rho)_p\geq p$,
      each $P$-orbit in $S(V)$ is non-trivial,
      i.e., $[P:P_x]\geq p$ for each $x\in S(V)$.
    \item[\rm(ii)$\implies$(i).]
      Suppose (ii) is true.
      Then, $p$ divides $\abs{P/P_x}=\abs{P(x)}$, which
      divides $\abs{G(x)}$, for any $x\in S(V)$.
      It follows that $p$ divides $\alpha(\rho)$.
    \item[\rm(ii)$\iff$(iii).]
      Both (ii) and (iii) are equivalent to $\dim{V^P}=0$.
    \end{description}
  \end{proof}

  \begin{remark}\label{rmk:solvable_group1}
    Notice that in \Cref{lemma:solvable_group1},
    for (ii) to imply (i),
    it is enough to assume that $P<G$ is a $p$-subgroup.
  \end{remark}
  
  \begin{remark}\label{rmk:solvable_group2}
    In what follows, denote $\underline{H}:={\ds\sum_{g\in H}}{g}\in\bz[G]$
    and $\hat{\underline{H}}:=\frac{1}{\abs{H}}\underline{H}\in\bq[G]$
    for any $H<G$.
    Under this notation, \Cref{lemma:solvable_group1} (iii)
    reads $\chi(\underline{P})=0$ or $\chi(\hat{\underline{P}})=0$.
    In addition, note that \Cref{lemma:solvable_group1} (iii)
    is equivalent to saying that $\rho$ is not a constituent of
    $\indrep{\trep{P}}{G}$.
  \end{remark}

  \begin{prop}\label{lemma:solvable_group2}
    Let $V$ be an non-trivial irreducible $G$-representation
    and $N\trianglelefteq G$.
    Then, $N$ acts non-trivially on $V$
    if and only if $N_x\lneq N$ for any $x\in S(V)$.
  \end{prop}

  \begin{proof}
    Since $N$ is normal in $G$, the subspace $V^N$ is $G$-invariant.
    Therefore, either $V^N=\set{\mathbf 0}$
    or $V^N=V$ from which the claim follows.
  \end{proof}

  \noindent
  The next result immediately follows from
  \Cref{lemma:solvable_group1} and \Cref{lemma:solvable_group2}
  (see also \Cref{rmk:solvable_group1}).

  \begin{coro}\label{coro:nameless2}
    Let $N\trianglelefteq G$ be a $p$-subgroup and let $V$ be a
    non-trivial irreducible $G$-representation
    where $N$ acts non-trivially.
    Then, $p$ divides $\alpha(\rho)$.
  \end{coro}

  As for sufficiency in \Cref{thm:solvable_groups},
  we will need the following lemma.

  \begin{lemma}\label{lemma:solvable_group3}
    Let $\rho$ be a $G$-representation with character $\chi$ and $H\leq G$.
    Then,
    \begin{description}[leftmargin=0pt]
    \item[\rm(i)] $\rho(\hat{\underline{H}})$ is an idempotent.
    \item[\rm(ii)] If, in addition, $\chi(\hat{\underline{H}})=0$,
      then both $\rho(\hat{\underline{H}})$ and 
      $\rho(\underline{H})$ are zero matrices.
    \end{description}
  \end{lemma}

  \begin{proof}
    Direct computation shows that
    $\hat{\underline{H}}\in\bq[G]$
    is an idempotent, therefore,
    so is $\rho(\hat{\underline{H}})$.
    If, in addition, $\chi(\hat{\underline{H}})=0$,
    i.e., $\rho(\hat{\underline{H}})$
    is an idempotent matrix with zero trace,
    then $\rho(\hat{\underline{H}})$ is a zero matrix.
    In this case,
    $\rho(\underline{H})=\abs{H}\rho(\hat{\underline{H}})$
    is also a zero matrix.
  \end{proof}

  \noindent
  The next result follows immediately from
  \Cref{lemma:solvable_group1,lemma:solvable_group3}
  (see also \Cref{rmk:solvable_group2}).

  \begin{coro}\label{coro:solvable_group}
    Let $\rho\in\irrnt(G)$ with $\alpha(\rho)>1$.
    Then, there exists a prime factor $p$ of $\abs{G}$
    such that $\rho(\underline{P})$ is a zero matrix
    for any $P\in\syl_p(G)$.
  \end{coro}

  \noindent
  The following elementary statement is an immediate consequence 
  of the injectivity of a regular representation of a finite group.

  \begin{prop}\label{prop:solvable_group4}
    Let $\rho$ be the regular $G$-representation. Given
    two elements $x,y$ of the group algebra $\bq[G]$,
    if $\rho(x)=\rho(y)$, then $x=y$.
  \end{prop}


  We are now in a position to prove \Cref{thm:solvable_groups}.

  \begin{proof}[Proof of \Cref{thm:solvable_groups}]
    \begin{description}[leftmargin=0cm]
    \item[Necessity.]
      We will prove the necessity by induction.
      Clearly, \eqref{eq:dagger} is true for $\abs{G}=1$.
      For the inductive step, assume that $\eqref{eq:dagger}$
      is true for solvable groups of order less than $m$.
      Suppose $\abs{G}=m$.
      Let $N$ be a minimal normal subgroup of $G$.
      Then, $N$ is a $p$-subgroup
      (see \Cref{prop:solvable_minimal_normal}).
      If $N=G$, then the result follows
      (see \Cref{example:basic_A_groups}).
      Otherwise, if $N\neq G$, consider
      an arbitrary $\rho\in\irrnt(G)$.
      If $N$ is not contained in the kernel of $\rho$,
      then $p$ divides $\alpha(\rho)$
      (see \Cref{coro:nameless2})
      and hence, $\alpha(\rho)>1$.
      If $N$ is contained in the kernel of $\rho$,
      then $\rho$ can be viewed as a non-trivial
      irreducible $(G/N)$-representation.
      Since $G/N$ is solvable and $\abs{G/N}<m$,
      by inductive assumption,
      $\alpha(\rho)>1$.

    \item[Sufficiency.]
      Assume \eqref{eq:dagger} is true.
      Then, for any $\rho\in\irrnt(G)$,
      there exists a prime divisor $p$ (depending on $\rho$)
      such that $\rho(\underline{P})$ is a zero matrix for
      any $P\in\syl_p(G)$ (see \Cref{coro:solvable_group}).
      Let $(p_1,\dots,p_k)$ be a sequence of all distinct
      prime divisors of $\abs{G}$ (no matter what the order is).
      Take an arbitrary collection of Sylow subgroups
      $\set{P_i:P_i\in\syl_{p_i}(G)}_{i=1}^k$.
      We claim that
      $\rho(\underline{G})=\rho(\mathcal{P})$,
      where $\mathcal{P}=\underline{P_1}\cdots\underline{P_k}$,
      for any $\rho\in\irr(G)$.
      Indeed,

      \begin{enumerate}
      \renewcommand\labelenumi{(\alph{enumi})}
      \item if $\rho$ is trivial, then
        $\rho(\mathcal{P})=\abs{G}=\rho(\underline{G})$;
      \item if $\rho$ is non-trivial, then
        since $\rho(\underline{P_j})$ is a
        zero matrix for some $1\leq j\leq k$
        (see \Cref{lemma:solvable_group1}, \Cref{rmk:solvable_group2} and
        \Cref{lemma:solvable_group3}),
        so is $\rho(\mathcal{P})$.
        On the other hand, since $\rho\in\irrnt(G)$, it follows that
        $\chi(\hat{\underline{G}})=\iprod{\chi}{\trep{G}}=0$ and
        therefore, $\rho(\underline{G})$ is also a zero matrix
        (see \Cref{lemma:solvable_group3}).
      \end{enumerate}
   
      Then, $\underline{G}=\mathcal{P}$
      (see \Cref{prop:solvable_group4}),
      from which it follows $G=P_1\cdots P_k$.
      Since $P_j$ is arbitrarily taken from $\syl_{p_j}(G)$,
      it follows from \Cref{thm:KaplanLevy} that $G$ is solvable.
    \end{description}
  \end{proof}

  \begin{example}
    Since $A_5$ is not solvable,
    there exists $\rho\in\irrnt(A_5)$
    such that $\alpha(\rho)=1$.
    According to \Cref{tbl:a5_chartbl}, this is the only
    non-trivial irreducible representation of $A_5$
    with trivial $\alpha$-characteristic.
  \end{example}

  \subsection{$\alpha$-characteristic of Nilpotent Group Representations}
  If $G$ is a nilpotent group, then one can
  strengthen the necessity part of \Cref{thm:solvable_groups}
  as follows.

  \begin{prop}\label{prop:nilpotent_realize_alpha}
    If $G$ is a nilpotent group, then
    for any $\rho\in\irr(G)$,
    \begin{align}\label{eq:ddagger}
      \mbox{there exists $v\in S(V)$ such that
      $\alpha(\rho)=\abs{G(v)}$}.\tag{$\ddagger$}
    \end{align}
  \end{prop}

  \noindent
  We say that $\alpha(\rho)$ is {\em realized by
  the orbit $G(v)$} or simply {\em realizable}
  if it satisfies \eqref{eq:ddagger}.
  The proof of \Cref{prop:nilpotent_realize_alpha}
  is based on the following two lemmas.

  \begin{lemma}\label{lemma:nilpotent1}
    Let $\rho:A\to\gl(V)$ and $\psi:B\to\gl(W)$
    be two $G$-representations.
    Then, 
    $\mathrm{lcm}(\alpha(\rho),\alpha(\psi))$ divides
    $\alpha(\rho\otimes\psi)$ and $\alpha(\rho\otimes\psi)$
    divides $\alpha(\rho)\alpha(\psi)$.
  \end{lemma}

  \begin{proof}
    Take $\tilde{A}:=A\times\trp<A\times B$.
    Then, $\resrep{(\rho\otimes\psi)}{\tilde{A}}$
    is equivalent to the direct sum of
    $\dim W$ copies of $\rho$ and hence,
    $\alpha(\resrep{(\rho\otimes\psi)}{\tilde{A}})=\alpha(\rho)$.
    Now \Cref{prop:alpha_func}(a) implies
    $\alpha(\rho)\mid\alpha(\rho\otimes\psi)$.
    Similarly, $\alpha(\psi)\mid\alpha(\rho\otimes\psi)$.
    It follows that
    $\mathrm{lcm}(\alpha(\rho),\alpha(\psi))\mid\alpha(\rho\times\psi)$.

    Let $v\in V,w\in W$ be non-zero vectors.
    Then, one has
    $A_v\times B_w\leq (A\times B)_{v\otimes w}\leq A\times B$,
    which implies that $\abs{(A\times B)(v\otimes w)}$ divides
    $\abs{A(v)}\cdot\abs{B(w)}$.
    Therefore, $\alpha(\rho\otimes\psi)$ divides
    every product $\abs{A(v)}\cdot\abs{B(w)}$.
    This implies that $\alpha(\rho\otimes\psi)$ divides
    \begin{align*}
      &\mathrm{gcd}\set{|A(v)|\cdot |B(w)|:v\in S(V),w\in S(W)}\\
      =\;&\mathrm{gcd}\set{\mathrm{gcd}\set{\abs{A(v)}:v\in S(V)}\cdot
          \abs{B(w)}:w\in S(W)}\\
      =\;&\mathrm{gcd}\set{|A(v)|:v\in S(V)}\cdot
          \mathrm{gcd}\set{|B(w)|:w\in S(W)}\\
      =\;&\alpha(\rho)\alpha(\psi).
    \end{align*}
  \end{proof}

  \begin{remark}
  If the orders $|A|$ and $|B|$ are coprime,
  then the numbers $\alpha(\rho),\alpha(\psi)$ are coprime too,
  and, therefore,
  $\mathrm{lcm}(\alpha(\rho),\alpha(\psi))=\alpha(\rho)\alpha(\psi)$
  implying $\alpha(\rho\otimes\psi)=\alpha(\rho)\alpha(\psi)$.
  If the group orders are not coprime,
  then it could happen that $\alpha(\rho\otimes\psi)$ satisfy 
  $\mathrm{lcm}(\alpha(\rho),\alpha(\psi))<\alpha(\rho\otimes\psi)
  <\alpha(\rho)\alpha(\psi) $.
  As an example, one could take $A=B$ to be an extra special group
  of order $p^3,p$ is an odd prime.
  This group has $p-1$ Galois conjugate representations of dimension $p$.
  Each of these representations is induced
  from a one-dimensional
  representation of $\bz_p\times\bz_p$, from which it follows
  that the $\alpha$-characteristic of each representation
  is equal to $p^2$.
  Let $\rho$ be one of these representations.
  Then, the irreducible repesentation $\rho\otimes\rho$
  of $A\times A$ has $\alpha$-characteristic equal to $p^3$.
  \end{remark}

  \begin{lemma}\label{lemma:nilpotent2}
    Let $\rho:A\to\gl(V)$ and $\psi:B\to\gl(W)$
    be two $G$-representations with
    $\gcd\set{\abs{A},\abs{B}}=1$.
    Then, $\abs{(A\times B)(v\otimes w)}=
    \abs{A(v)}\cdot\abs{B(w)}$
    for any $v\in S(V)$ and $w\in S(W)$.
  \end{lemma}

  \begin{proof}
    Suppose $\gcd\set{\abs{A},\abs{B}}=1$.
    Note that by definition,
    \begin{align*}
      (A\times B)(v\otimes w)=A(v)\otimes B(w)
          :=\set{av\otimes bw:a\in A,\;b\in B}
    \end{align*}
    for any $v\in V$ and $w\in W$.
    Hence, it suffices to show that
    the map $(av,bw)\mapsto av\otimes bw$
    is an injection from $A(v)\times B(w)$
    to $A(v)\otimes B(w)$ for
    any $v\in S(V)$ and $w\in S(W)$.

    Suppose that $a_1v\otimes b_1w=a_2v\otimes b_2w$
    for some $a_1,a_2\in A$ and $b_1,b_2\in B$.
    This is true if and only if
    \begin{align*}
      a_2v=\lambda a_1v\quad\mbox{and}\quad
          b_2w=\lambda^{-1} b_1w
    \end{align*}
    or, equivalently,
    \begin{align*}
      {a_1}^{-1}a_2v=\lambda v\quad\mbox{and}\quad
          {b_2}^{-1}b_1w=\lambda w
    \end{align*}
    for some $\lambda\in\bc$.  Hence,
    $\ord(\lambda)$ divides both
    $\ord({a_1}^{-1}a_2)$ and $\ord({b_2}^{-1}b_1)$,
    which are factors of $\abs{A}$ and $\abs{B}$,
    respectively.
    It follows that $\ord(\lambda)=1$ and hence, $\lambda=1$.
    Therefore, $a_1u=a_2u$ and $b_1u=b_2u$ and the result follows.
  \end{proof}

  We are now in a position to prove
  \Cref{prop:nilpotent_realize_alpha}.

  \begin{proof}[Proof of \Cref{prop:nilpotent_realize_alpha}]



    We use induction on the number $k$ of distinct prime divisors
    of $\abs{G}$, i.e., $\abs{G}=\prod_{i=1}^k{{p_i}^{j_i}}$.

    If $k=1$, then $G$ is a $p$-group and, therefore,
    the length of every $G$-orbit is a power of $p$.
    Hence, $\alpha(\rho)=\min_{v\in S(V)}\set{\abs{G(v)}}$.

    Assume now that $k\geq 2$ and $p_1,\dots,p_k$ are the
    prime divisors of $\abs{G}$. Then, the group $G$ is a
    direct product of its Sylow subgroups $P_i$, $i=1,\dots,k$,
    where $P_i$ is a Sylow $p_i$-subgroup of $G$.
    Therefore, $G\cong G_1\times G_2$ where
    $G_1=P_1$ and $G_2=P_2\times\dots\times P_k$.
    Pick an arbitrary $\rho\in\irr(G)$.
    Then, $\rho$ is equivalent to $\psi_1\otimes\psi_2$
    where $\psi_1\in\irr(G_1)$ and $\psi_2\in\irr(G_2)$
    (see \cite{Serre}).

    Hence, it suffices to show that
    $\alpha(\psi_1\otimes\psi_2)$ is
    realizable for any $\psi_j\in\irr(G_j)$
    ($\psi_j:G_j\to\gl(V_j)$), $j=1,2$.
    By inductive assumption,
    $\alpha(\psi_j)=\abs{G_j(v_j)}$
    for some $v_j\in S(V_j)$, $j=1,2$.
    Since $\gcd\set{\abs{G_1},\abs{G_2}}=1$,
    \begin{align}\label{eq:nilpotent_realize_alpha_proof1}
      \abs{(G_1\times G_2)(v_1\otimes v_2)}
          =\abs{G_1(v_1)}\cdot\abs{G_2(v_2)}
          =\alpha(\psi_1)\cdot\alpha(\psi_2)
    \end{align}
    (see \Cref{lemma:nilpotent2}). On the other hand,
    one has
    \begin{align}\label{eq:nilpotent_realize_alpha_proof2}
      \alpha(\psi_1)\cdot\alpha(\psi_2)\divides
      \alpha(\psi_1\otimes\psi_2)\divides
      \abs{(G_1\times G_2)(v_1\otimes v_2)}
    \end{align}
    (see \Cref{lemma:nilpotent1}).
    Combining \eqref{eq:nilpotent_realize_alpha_proof1}
    and \eqref{eq:nilpotent_realize_alpha_proof2}
    yields
    \begin{align*}
      \alpha(\psi_1\otimes\psi_2)
        =\abs{(G_1\times G_2)(v_1\otimes v_2)}.
    \end{align*}

    \noindent
    The result follows.
  \end{proof}

  \begin{example}
    In general, for a solvable group $G$,
    $\alpha(\rho)$ may not be realizable for
    some $\rho\in\irr(G)$:
    Let $\rho$ be the 3-dimensional
    irreducible $A_4$-representation.
    Then, $\Phi(\rho)= \set{(\bz_1),(\bz_2),(\bz_3)}$,
    and $\alpha(\rho)=2$ is not realizable.
  \end{example}
    
  \begin{example}
    If $G$ is a $p$-group,
    then $\alpha(\rho)$ is realizable for
    {\em any} $G$-representation $\rho$,
    which is not the case if $G$ is a
    nilpotent group but not a $p$-group.
    In fact, consider the (reducible) $\bz_6$-representation
    $\rho:=\psi_1\otimes\trep{\bz_3}
    \oplus\trep{\bz_2}\otimes\psi_2$,
    where $\psi_1$ and $\psi_2$ are arbitrary
    non-trivial irreducible representations
    of $\bz_2$ and $\bz_3$, respectively.
    Then, $\Phi(\rho)=\set{(\bz_1),(\bz_2),(\bz_3)}$,
    and $\alpha(\rho)=1$ is not realizable.
  \end{example}

  \section{$\alpha$-characteristic of an augmentation module related to $2$-transitive Group actions}
  \label{sec:4}

  The following example is the
  starting point of our discussion.

  \begin{example}\label{example:a5}
    Note that $A_5$, the smallest {\em non-solvable} group,
    is a $2$-transitive group (see \Cref{def:2trans_group,def:aug_subm}).
    To be more explicit, $A_5$ admits two non-equivalent irreducible
    augmentation representations $\psi_3=\rho^a_{[A_5;A_4]}$
    and $\psi_4=\rho^a_{[A_5;D_5]}$
    (see \Cref{tbl:a5_chartbl}).
    Since $\alpha(\psi_3)=5$ while $\alpha(\psi_4)=1$,
    we arrive at the following question:
    given an augmentation submodule $\rho\in\varrho_2(G)$
    associated to the $2$-transitive $G$-action on $G/H$
    by left translation, under which conditions
    does one have $\alpha(\rho)>1$?
  \end{example}

  \subsection{2-transitive Groups}
  \label{subsec:2trans_groups}
  Let $G$ be a $2$-transitive group acting faithfully on a set $X$.
  According to Burnside Theorem (see \cite[Theorem 4.1B]{DM}),
  the socle $S$ of $G$ is either a non-abelian simple group or
  an elementary abelian group which acts regularly on $X$.
  Thus, 2-transitive groups are naturally divided into two classes

  \begin{itemize}
\item {\em Almost simple groups.}
    $G$ is called {\em almost simple} if
    $S\leq G\leq\aut(S)$ for some
    non-abelian simple group $S$.
  \item {\em Affine groups.}
    If $S$ is elementary abelian, then $G$ admits the following description.

    Let $V$ be a $d$-dimensional vector space over
    a finite field $\ff$. A group
    $G$ is called {\em affine} if
    $V\leq G\leq\aff(V)$, where
    $V$ is considered as an additive group
    and $\aff(V)$ is the group of all
    invertible affine transformations
    of $V$. A group $G$ admits a decompostion
    $G= V G_o$ where $G_o = G\cap\gl(V) $ is a zero stabilizer in $G$.
    Thus, $G\cong V\rtimes G_o$  ($\rtimes$ stands
    for the semidirect product).
    The group $G$ acts 2-transitively on $V$ if and only if $G_o$
    acts transitively on the set of non-zero vectors of $V$.
    In this case, $V$ is the socle of $G$.
  
  \end{itemize}


  \begin{remark}\label{rmk:2trans1}
  Note that a solvable $2$-transitive group is always affine.
  However, the converse is not true: for example, the full affine group
  $\aff(V)$ of the vector space $V$ is solvable if and only if $\gl(V)$ is.
  The latter happens only when $d=1$ or $d=2$ and $\abs{\ff}\in\{2,3\}$.
  \end{remark}

  \subsection{Main Result}
  Our main result provides a complete description of
  all augmentation modules related to 2-transitive 
  group actions with non-trivial $\alpha$-characteristic.  

  \begin{thm}\label{thm:2trans_groups}
    Let $(G;X)$ be a $2$-transitive group action.
    \begin{enumerate}
    \renewcommand{\labelenumi}{(\roman{enumi})}
    \item If $G$ is affine and acts 
      faithfully on $X$, then $\alpha(\rho_{(G;X)}^a)>1$.
    \item If $G$ is almost simple, then
      all $\rho\in\varrho_2(G)$ satisfying
      $\alpha(\rho)>1$ are described in
      \Cref{tbl:almost_simple_2trans}
      provided that $\abs{X}$ is a prime power.
    \end{enumerate}
  \end{thm}

    \begin{table}[h!]
    \tiny
    \centering
    \begin{tabularx}{\textwidth}{@{}lllll@{}}
      \toprule
      $\abs{X}$ & condition & $N$ & $\max\abs{G/N}$ &
          \# of non-equivalent actions \\
      \midrule
      $2$ & $n\geq 5$ & $A_n$ & $2$ & $1$ \\
      $n$ & $n\geq 5$ & $A_n$ & $2$ & $4$ if $n=6$; $2$ otherwise \\
      $(q^d-1)/(q-1)$ & $d\geq2$, $(d,q)\neq(2,2),(2,3)$ &
          $\psl(d,q)$ & $\gcd\set{d,q-1}\cdot e$ &
          $2$ if $d>2$; $1$ otherwise \\
      $11$ & & $\psl(2,11)$ & $1$ & $2$ \\
      $11$ & & $M_{11}$ & $1$ & $1$ \\
      $23$ & & $M_{23}$ & $1$ & $1$ \\
      \midrule
      \multicolumn{5}{l}{$q=p^e$ for some prime $p$ and $e\in\bn$;
          $N$ is the socle of $G$;
          $M_{11}$ and $M_{23}$ are Mathieu groups} \\
      \bottomrule
    \end{tabularx}
    \caption{almost simple $2$-transitive
        groups with $\rho\in\varrho_2(G)$
        satisfying $\alpha(\rho)>1$}
    \label{tbl:almost_simple_2trans}
  \end{table}

  \noindent
  The proof of \Cref{thm:2trans_groups} is
  based on the classification of finite
  2-transitive groups (see \cite{Cameron})
  and the following lemma.

  \begin{lemma}\label{lemma:2trans_groups}
    Let $H$ be a proper subgroup of $G$.
    Then, $\alpha(\rho_{[G;H]}^a)>1$
    if and only if $\subgind{G}{H}$
    is a prime power. In the latter case,
    $p$ divides $\alpha(\rho_{[G;H]}^a)$ where
    $p$ is the unique prime divisor of $\subgind{G}{H}$.
  \end{lemma}

  \begin{proof}
    If $[G:H]$ is not a prime power, then for each prime divisor
    $p$ of $|G|$ a Sylow subgroup $P\in\syl_p(G)$ has at
    least two orbits on $G/H$. Therefore, $[(1_H^G)_P,\trep{P} ]\geq2$
    implying $[\rho_{[G;H]}^a,\trep{P}]>0$.
    By \Cref{lemma:solvable_group1},
    $p$ does not divide $\alpha(\rho_{[G;H]}^a)$.
    Since this holds for any prime divisor of $|G|$, 
    we conclude that $\alpha(\rho_{[G;H]}^a) = 1$.

    Conversely, suppose $\abs{G:H}=p^k$ for some prime $p$.
    Then, a Sylow $p$-subgroup $P\in\syl(G)$ acts transitively
    on the coset space $G/H$ implying that
    $[\trep{H}^G,\trep{P}]=1$. Therefore 
    $[\rho_{[G;H]}^a,\trep{P}]=0$ and we are done by
    \Cref{lemma:solvable_group1}.
  \end{proof}



%

 \begin{remark}\label{rmk:2trans_groups}
    Note that although any non-trivial
    irreducible representation of a solvable group admits
    a non-trivial $\alpha$-characteristic (see \Cref{thm:solvable_groups}),
    it may not be true for their direct sum
    (see \Cref{prop:alpha_func}(e)).
    However, by \Cref{lemma:2trans_groups}, an augmentation
    submodule associated to a transitive $G$-set of order prime power
    would admit non-trivial $\alpha$-characteristic. In particular,
    the $\alpha$-characteristic of every non-trivial irreducible
    constituent of the augmentation submodule is non-trivial.
  \end{remark}

  Now we can prove the aforementioned \Cref{thm:2trans_groups}.

  \begin{proof}[Proof of \Cref{thm:2trans_groups}]
    If $G$ is an affine $2$-transitive group,
    then its socle $S$ is an elementary abelian group which
    acts faithfully on $X$. By the Burnside Theorem,
    $S$ acts regularly on $X$. Therefore, $|X|=|S|$ is a prime power
    and we are done by \Cref{lemma:2trans_groups}.
   
    If $G$ is an almost simple
    $2$-transitive group, then all $2$-transitive
    $G$-sets of power prime order are obtained by the inspection
    of Table 7.4 from \cite{Cameron}, which yields
    \Cref{tbl:almost_simple_2trans}.
  \end{proof}

  \begin{remark}
    For the complete description of 2-transitive faithful
    actions of affine groups, we refer to Table 7.3
    in \cite{Cameron}.
  \end{remark}


  \subsection{Examples}
  In this subsection, we will give some
  concrete examples of $2$-transitive groups
  illustrating \Cref{thm:2trans_groups} and
  \Cref{rmk:2trans_groups}.

  \begin{example}\label{ex:gl2}
    The group $G:=\mathrm{AGL}_3(2) = \Z_2^3\rtimes\gl(3,2)$
    is a non-solvable 2-transitive affine group
    (see \Cref{rmk:2trans1})
    with four augmentation representations (see \Cref{tbl:gl_chartbl})
    arising from 2-transitive actions.
    By \Cref{thm:2trans_groups} (i),
    $\alpha(\rho)>1$ for all $\rho\in\varrho_2(G)$.
  \end{example}

  \begin{example}\label{ex:s5}
    The group $S_5$ is an almost simple group with three
    2-transitive actions: $S_5/A_5$, $S_5/S_4$ and $S_5/\aff_1(5)$.
    The corresponding augmentation representations are denoted by
    $\xi_1$, $\xi_3$ and $\xi_6$ in \Cref{tbl:s5_chartbl}.
    According to Lemma \ref{lemma:2trans_groups},
    only $\xi_1=\rho^a_{[S_5;A_5]}$ and
    $\xi_6=\rho^a_{[S_5;S_4]}$ admit non-trivial
    $\alpha$-characteristics.
  \end{example}

  \begin{remark}\label{rmk:s5}
    In some cases, \Cref{thm:2trans_groups}
    can still help one to determine whether $\alpha(\rho)$
    is trivial even when $\rho$ is
    not an augmentation representation.
    For example, $S_5$ admits
    two 4-dimensional irreducible
    representations $\xi_2$ and $\xi_6$
    (see \Cref{tbl:s5_chartbl}).
    Note that $\xi_6$ is an augmentation
    representation related to a $2$-transitive action while $\xi_2$ is not.
    Let $V$ and $V^-$ be $S_5$-modules corresponding to
    $\zeta_6$ and $\zeta_2$, respectively. In \Cref{sec:6}, we will
    show that there exists an admissible
    equivariant map from $V^-$
    to $V$ from which it follows that
    $\alpha(\xi_2)\geq\alpha(\xi_6)>1$
    (see \Cref{ex:s5})---this agrees with
    \Cref{tbl:s5_chartbl}.
    One can apply similar argument to $(n-1)$-dimensional
    irreducible $S_n$-representations with
    $n>5$ being a prime power.
  \end{remark}

  \noindent
  Our last example illustrates \Cref{rmk:2trans_groups}.

  \begin{example}
    Consider the solvable group $G:=SL_2(\mathbb{Z}_3)$
    acting transitively (but not $2$-transtively,
    in particular, the augmentation representation is reducible,
    see \Cref{def:aug_subm})
    on the set $X$ of eight non-zero vectors of
    $(\mathbb{Z}_3)^2$. It follows from
    \Cref{lemma:2trans_groups} that
    $2\mid\alpha(\rho_{[G;X]}^a)$.
    Therefore, $2$ divides $\alpha$-characteristic
    of every non-trivial constituent of 
    $\rho_{[G;X]}^a$, and there are three of those:
    two $2$-dimensional and one $3$-dimensional.
  \end{example}

  \section{Irreducible representations with trivial $\alpha$-characteristic}
  \label{sec:5}
  As we already know, a finite group $G$ admitting an irreducible
  complex representation $\rho$ with trivial $\alpha$-characteristic
  is non-solvable. 

  In this section, given $N\triangleleft G$ and
  an irreducible $G$-representation (resp.~irreducible $N$-representation)
  with trivial $\alpha$-characteristics,
  we study the $\alpha$-characteristics of
  its restriction to $N$ (resp.~induction to $G$).

  \subsection{Motivating Examples}
  Keeping in mind \Cref{prop:alpha_func},
  consider the following example.

  \begin{example}\label{example:rest_to_normal}
    Consider $N:=A_5\trianglelefteq S_5=:G$
    ($N$ is simple and $G\simeq\aut(N)$).
    The isotypical decomposition of
    $\resrep{\xi}{N}$ for $\xi\in\irr(G)$
    are as follows
    (see also \Cref{tbl:a5_chartbl,tbl:s5_chartbl}):
    \begin{align*}
      \resrep{(\xi_0)}{N}=\resrep{(\xi_1)}{N}&=\psi_0,\\
      \resrep{(\xi_2)}{N}=\resrep{(\xi_6)}{N}&=\psi_3,\\
      \resrep{(\xi_3)}{N}=\resrep{(\xi_5)}{N}&=\psi_4,\\
      \resrep{(\xi_4)}{N}&=\psi_1\oplus\psi_2.
    \end{align*}
    Observe that
    \begin{enumerate}
    \renewcommand{\labelenumi}{(\roman{enumi})}
    \item $\alpha(\psi_0)=1$ divides both
      $\alpha(\xi_0)=1$ and $\alpha(\xi_1)=2$;
    \item $\alpha(\psi_3)=5$ divides both
      $\alpha(\xi_2)=10$ and $\alpha(\xi_6)=5$;
    \item $\alpha(\psi_4)=1$ divides both
      $\alpha(\xi_3)=1$ and $\alpha(\xi_5)=1$;
    \item both $\alpha(\psi_1)=2$ and $\alpha(\psi_2)=2$ divide
      $\alpha(\xi_4)=2$.
    \end{enumerate}
  \end{example}

  \begin{remark}\label{rmk:5.1}
    Clearly, \Cref{example:rest_to_normal} is
    in the complete agreement with \Cref{prop:alpha_func}
    (b).  On the other hand, it also gives
    rise to the following question: under which condition,
    does $\alpha(\theta)=1$ imply $\alpha(\indrep{\theta}{G})=1$
    for $\theta\in\irr(N)$ and $N\trianglelefteq G$?
    The answer is given in the nest subsection. 
  \end{remark}


  \subsection{Induction and restriction of representations with
  trivial $\alpha$-characteristic}

  Let $N$ be a non-trivial proper subgroup of $G$.
  Pick an arbitrary $\rho\in\irr(G)$.
  By \Cref{prop:alpha_func} (a), $\alpha(\rho_N)$ divides $\alpha(\rho)$.
  Therefore, if $\alpha(\rho)=1$, then $\alpha(\rho_N) = 1$.
  Decomposing $\rho_N$ into a direct sum of $N$-irreducible
  representations $\rho_N=\sum_{i=1}^k\theta_i$,
  we obtain $\gcd(\alpha(\theta_1),...,\alpha(\theta_k))=1$
  (see \Cref{prop:alpha_func} (e)).
  If $N$ is normal in $G$, then all constituents $\theta_i$
  are $G$-conjugate by Clifford's theorem (see \cite{Isaacs1976}),
  and have the same $\alpha$-characteristic (see \Cref{prop:alpha_func} (c)).
  Hence, $\alpha(\theta_i)=1$ for each $i=1,...,k$.
  In other words, trivial $\alpha$-characteristic of an irreducible
  $G$-representation $\rho$ is inherited by all constituents of its
  restriction $\rho_N$.
  But this does not happen for the induction.
  More precisely, if $\theta\in\irrnt(N)$
  has trivial $\alpha$-characteristic,
  then some of the constituents of $\theta^G$ may
  have non-trivial $\alpha$-characteristic even in the case of
  $N$ being normal.
  For example, take $G=A_5\times\Z_7, N=A_5$ and $\theta =\rho^a_{A_5:D_5}$.
  Then $\theta^G=\sum_{i=0}^6 \theta\otimes\lambda^i$
  where $\lambda\in\irrnt(\Z_7)$.
  Clearly, $\alpha(\theta\otimes 1_{\Z_7})=\alpha(\theta)=1$.
  By \Cref{lemma:nilpotent1},
  $\alpha(\theta\otimes\lambda^i)=7$ if $i\neq 0$.
  Thus, $\theta^G$ contains only one irreducible constituent
  with trivial $\alpha$-characteristic.

  \begin{prop}\label{090517a} Let $N\trianglelefteq G$
  and $\theta\in\irr(N)$ with $\alpha(\theta)=1$.
  Then, $\alpha(\theta^G)=1$.
  \end{prop}

  \begin{proof}
  Let $W$ (resp.~$V$) be the $N$-representation
  (resp.~$G$-respresentation)
  corresponding to $\theta$ (resp.~$\theta^G$).
  It suffices to show that $\alpha(\theta^G)_p=1$
  for each prime divisor $p$ of $|G|$.

  Pick a Sylow $p$-subgroup $P\leq G$.
  Then, $P\cap N$ is a Sylow $p$-subgroup of $N$.
  It follows from $\alpha(\theta)=1$ that the subspace
  $W_1:=W^{P\cap N}$ is non-trivial (see \Cref{lemma:solvable_group1}).
  Pick an arbitarary non-zero $w\in W_1$.
  Then, the vector $v:=\sum_{g\in P} gw$ is fixed by any element of $P$,
  that is, $Pv=v$. We claim that $v\neq 0$.
  Let $T_1$ be a transversal of $P/(P\cap N)$.
  By isomorphism $P/(P\cap N)\cong PN/N$,
  the set $T_1$ is a transversal of $PN/N$.
  Now we complete $T_1$ to a transversal $T$ of $G/N$
  and set $V=\oplus_{t\in T} tW$.
  
  Now $P = T_1(P\cap N)$ implies
  $v=|P\cap N|\sum_{t\in T_1} t w\neq 0$.
  Thus, $V^P$ is non-trivial, and, consequently, $\alpha(\theta^G)_p=1$.
  \end{proof}

  In general, it is not clear whether $\alpha(\theta)=1$
  implies that $\theta^G$ contains an irreducible constituent
  with trivial $\alpha$-characteristic.
  \Cref{prop:normal} below provides sufficient conditions for that.
  Its proof is based on the following lemma (see \cite{Isaacs1976}).

  \begin{lemma}\label{lemma:clifford}
    Suppose $N\trianglelefteq G$. Let
    $\chi$ and $\omega$ be irreducible
    characters of $G$ and $N$, respectively,
    such that $\iprod{\resrep{\chi}{N}}{\omega}>0$.
    If $p=\subgind{G}{N}$ is a prime,
    then for the decompositions of
    $\resrep{\chi}{N}$ and $\indrep{\omega}{G}$,
    one of the following two statements takes place:

    \begin{enumerate}
    \renewcommand\labelenumi{(\alph{enumi})}
    \item $\chi_N=\ds\sum_{i=0}^{p-1}\gconj{\omega}{g_i}$
      and $\indrep{\omega}{G}=\chi$.
    \item $\chi_N=\omega$ and
      $\indrep{\omega}{G}=\ds\sum_{i=0}^{p-1}{\chi\phi_i}$,
      where $\set{\phi_i}_{i=0}^{p-1}$ is the set of
      all irreducible characters of $G/N\simeq\bz_p$,
      which can be viewed as irreducible characters of $G$ as well.
    \end{enumerate}
  \end{lemma}

  \begin{prop}\label{prop:normal}
    Let $N\trianglelefteq G$ and $\omega\in\irr(N)$ with  $\alpha(\omega)=1$.
    If $G/N$ is solvable then, $\indrep{\omega}{G}$ admits
    an irreducible component $\rho$
    satisfying $\alpha(\rho)=1$.
  \end{prop}

  \begin{proof}
    We use induction over $|G/N|$.
    Pick a maximal normal subgroup $M$ of $G$
    which contains $N$. Then, $M/N$ is a maximal normal subgroup of $G/N$,
    and, by solvability of $G/N$, $M$ is of prime index, say $p$, in $G$.
    If $M\neq N$, then, by induction hypothesis,
    the representation $\omega^M$ has an irreducible component,
    say $\sigma$, with $\alpha(\sigma)=1$. 
    Applying induction hypothesis to the pair $M\trianglelefteq G$
    and $\sigma$ we conclude that $\sigma^G$ contains an irreducible
    component $\rho\in\irr(G)$ with $\alpha(\rho)=1$.
    Now it follows from $\omega^G = (\omega^M)^G$ that
    $\rho$ is a constituent of $\omega^G$.

    Assume now that $M=N$, that is $[G:N]=p$ is prime.
    By Lemma~\ref{prop:normal} either $\omega^G$ is irreducible or 
    $\omega^G = \sum_{i=0}^{p-1} \chi \xi_i$
    (hereafter, $\omega,\chi,\zeta_j$ stand for
    characters rather than for representations).
    In the first case we are done by \Cref{090517a}.
    Consider now the second case:

    $\indrep{\omega}{G}=\ds\sum_{i=0}^{p-1}{\chi\xi_i}$.
    In this case, in suffices to show that
    there exists $0\leq j\leq p-1$ such that
    $\chi\xi_j(\underline{R})>0$
    for any prime factor $r$ of $\abs{G}$ and
    $R\in\syl_r(G)$. Indeed,
    note that $\ds\sum_{i=0}^{p-1}\xi_i(g)=p$
    if $g\in N$ and $0$ otherwise. Therefore,
      \begin{align*}
        \indrep{\omega}{G}(\underline{R})=
        \sum_{i=0}^{p-1}\chi\xi_i(\underline{R})=\;&
        \sum_{g\in R}\sum_{i=0}^{p-1}\chi(g)\xi_i(g)=
        \sum_{g\in R}\chi(g)\left(\sum_{i=0}^{p-1}\xi_i(g)\right)\\
        =\;&
        p\sum_{g\in R\cap N}\chi(g)=
        p\sum_{g\in R\cap N}\omega(g)=
        p\,\omega(\underline{R\cap N}).
      \end{align*}
      Since $R\cap N$ is either trivial
      or a Sylow $r$-subgroup of $N$,
      one always has
      \begin{align*}
        \sum_{i=0}^{p-1}\chi\xi_i(\underline{R})=
        p\,\omega(\underline{R\cap N})>0.
      \end{align*}
      If $r=p$, then $\chi\xi_j(\underline{R})>0$
      for some $0\leq j\leq p-1$.
      If $r\neq p$, then $R\leq N$ and it follows that
      $\chi\xi_j(\underline{R})=\chi(\underline{R})>0$
      for the same $j$ as well. The result follows.
  \end{proof}

  \noindent
  The following example illustrates \Cref{prop:normal}.
  
  \begin{example}
    Consider $N:=\mathrm{PSL}(2,8)\trianglelefteq
    \mathrm{P\Gamma L}(2,8)=:G$
    ($N$ is simple, $G\simeq\aut(N)$ and $\subgind{G}{N}=3$
    is a prime).
    The relation between $\irr(G)$ and $\irr(N)$
    with respect to restriction and induction
    are described as follows
    (see also \Cref{tbl:psl28_chartbl,tbl:pgammal28_chartbl}):

    \noindent
    \begin{minipage}[c][2.75cm][c]{.8\textwidth}
    \centering
    \begin{tabularx}{.8\textwidth}{@{}R@{\,}LCR@{\,}L@{}}
      \resrep{\xi_0}{N}=\resrep{\xi_1}{N}
          =\resrep{\xi_2}{N}&=\psi_0,& &
          \indrep{\psi_0}{G}&=\xi_0\oplus\xi_1\oplus\xi_2,\\
      \resrep{\xi_3}{N}=\resrep{\xi_4}{N}
          =\resrep{\xi_5}{N}&=\psi_1,& &
          \indrep{\psi_1}{G}&=\xi_3\oplus\xi_4\oplus\xi_5,\\
      \resrep{\xi_6}{N}=\resrep{\xi_7}{N}
          =\resrep{\xi_8}{N}&=\psi_5,&and&
          \indrep{\psi_5}{G}&=\xi_6\oplus\xi_7\oplus\xi_8,\\
      \resrep{\xi_9}{N}&=\psi_2\oplus\psi_3\oplus\psi_4,& &
          \indrep{\psi_2}{G}=\indrep{\psi_3}{G}=\indrep{\psi_4}{G}
          &=\xi_9,\\
      \resrep{\xi_{10}}{N}&=\psi_6\oplus\psi_7\oplus\psi_8.& &
          \indrep{\psi_6}{G}=\indrep{\psi_7}{G}=\indrep{\psi_8}{G}
          &=\xi_{10}.
    \end{tabularx}
    \end{minipage}

    \noindent
    One can observe that
    $\alpha(\psi_6)=\alpha(\psi_7)
    =\alpha(\psi_8)=\alpha(\xi_{10})=1$,
    which agrees with \Cref{prop:normal}.
  \end{example}

  \subsection{Groups with totally trivial $\alpha$ characteristic}

  This section is devoted to the groups which have no irreducile
  representation with non-trivial $\alpha$-characteristic.
  In what follows we denote this class of groups as $\mathfrak{T}$.
  We know that this class is non-empty, since $J_1\in\mathfrak{T}$.
  We also know that all groups in $\mathfrak{T}$ are non-solvable.
  Below we collect elementary properties of $\mathfrak{T}$. 

  \begin{prop}\label{100517a} Let $G\in\mathfrak{T}$. Then
  \begin{enumerate}
  \item[(a)] If $N$ is a normal subgroup of $G$,
    then $N,G/N\in \mathfrak{T}$;
  \item[(b)] All composition factors of $G$ belong to $\mathfrak{T}$;
  \item[(c)] If $H\in\mathfrak{T}$, then $G\times H\in\mathfrak{T}$
  \end{enumerate}
  \end{prop}

  \begin{proof}
  Part (a). The inclusion $G/N$ follows from the fact that every
  irreducile representation of $G/N$ may be considered as an irreducile
  representation of $G$.
  The inclusion $N\in\mathfrak{T}$ follows from
  \Cref{prop:alpha_func} (b).

  Part (b) is a direct consequence of (a).

  Part (c) is a direct consequence of Lemma~\ref{lemma:nilpotent1},
  since each irreducible representation of $G\times H$ is a tensor product
  $\psi\otimes\phi$ where $\psi\in\irr(G),\phi\in\irr(H)$
  \end{proof}

  Our computations in GAP suggest the following conjecture.
  \medskip

  \noindent
  {\bf Conjecture.}
  If $G\in\mathfrak{T}$ is a simple group,
  then it is one of the sporadic groups.

  \section{Existence of Quadratic Equivariant Maps}\label{sec:6}
  
  \subsection{General Construction}
  In general, the problem of existence of
  equivariant maps between $G$-manifolds
  is rather complicated.
  We will study \Cref{prob:B} in the following
  setting: $V$ is a {\em faithful}
  {\em irreducible} $G$-representation
  and $W$ is another $G$-representation
  of the same dimension.
  It is well-known that
  if $V$ is faithful, then
  there exists a positive integer $k$ such
  that the symmetric tensor power
  $\sym^k(V)$ contains $W$
  (see, for example, \cite{Isaacs1976}). Thus, assume
  $W\subset\sym^k(V)$ and let
  \begin{align*}
    \begin{cases}
      \triangle:V\to V^{\otimes k},\\
      \triangle(v)=\underbrace{v\otimes
          v\otimes\dots\otimes v}_{k\mbox{ times}}.
    \end{cases}
  \end{align*}
  be the corresponding {\em diagonal} map.
  Observe that
  $\triangle$ is $G$-equivariant and
  $\triangle(V)$ spans $\sym(V)$.
  Let $A:\sym^k(V)\to W$ be
  a $G$-equivariant linear operator
  (e.g.~orthogonal projection).
  Then, $\phi=A\circ\triangle$
  is a $k$-homogeneous
  $G$-equivariant map from
  $V$ to $W$, which admits
  the following criterion
  of admissibility (see, for example, \cite{Kushkuley-Balanov}).
 
  \begin{prop}\label{prop:admissible}
    $\phi$ is admissible if and only if
    $\ker A\cap\triangle(V)=\set{\mathbf 0}$.
  \end{prop}

  \noindent
  In practice, given characters of $V$ and $W$,
  constructing an admissible
  homogeneous equivarinat map $\phi:V\to W$
  involves the following steps:

  \begin{enumerate}
  \renewcommand\labelenumi{(S\arabic{enumi})}
  \item\label[step]{step:A}
    Finding $k$ satisfying
    $W\subset\sym^k(V)$.

  \item\label[step]{step:B}
    Computing matrices representing
    the $G$-actions on $V$ and $\sym^k(V)$,
    and finding isotypical basis of $W\subset\sym^k(V)$.

  \item\label[step]{step:C}
    Verification of the admissibility
    of $\phi=A\circ\triangle$ by the
    Weak Nullstellensatz (see \cite{Cox-Little-Oshea}):
  \end{enumerate}

  \begin{prop}\label{prop:null}
    Let $Q:=\set{q_i}_{i=1}^m\subset R:=\bc[x_1,\dots,x_N]$
    be a collection of polynomials and
    $I:=\set{\sum_{i=1}^m{r_iq_i}:r_i\in R}$
    the ideal generated by $Q$.
    Then, the following statements are equivalent:

    \begin{itemize}
    \item $Q$ admits no common zeros.

    \item The Gr\"obner basis of $I$ contains
      the constant polynomial $1$.
    \end{itemize}
  \end{prop}

  \begin{remark}\
    \begin{enumerate}
    \renewcommand{\labelenumi}{(\roman{enumi})}
    \item \Cref{step:A,step:B}
      are related to the classical
      {\em Clebsch-Gordan problem} of an isotypical
      decompostion of tensor product of representations
      (see, for example, \cite{Cornwell}).

    \item For \Cref{step:C},
      one can use Mathematica
      (see \cite{mathematica})
      to compute the Gr\"obner basis.

    \item If $\phi$ is admissible,
      then $\deg(\phi)$ is
      well-defined and equal to $k^n$,
      where $n=\dim(V)$.
    \end{enumerate}
  \end{remark}

  To illustrate \Cref{prop:admissible} and
  also give a brief idea about
  \Crefrange{step:A}{step:C},
  consider the following example (see also
  \cite{Balanov-Krawcewicz-Kushkuley}
  and \cite{AED}); the detail about
  these steps will be provided later
  (see \Cref{subsection:s5example}).
  In what follows,
  denote by $(V,G)$ a faithful $G$-representation
  and by $(V,G/H)$ a non-faithful
  $G$-representation with
  kernel $H\trianglelefteq G$.

  \begin{example}\label{example:q8}
    Consider the symmetric tensor square
    $(\sym_\bc^2(\bh),\qb/\bz_2)$ of the complex representation
    $(\bh,\qb)$.
    One can easily check that
    \begin{align}
      e_1=\frac{1\otimes 1+j\otimes j}{2},\quad
      e_2=\frac{1\otimes 1-j\otimes j}{2},\quad
      e_3=\frac{j\otimes 1+j\otimes 1}{2}
    \end{align}
    form an isotypical basis of $\sym_\bc^2(\bh)$ and
    \begin{align}
      \triangle(z_1+jz_2)=e_1({z_1}^2+{z_2}^2)+
      e_2({z_1}^2-{z_2}^2)+e_3(2z_1z_2).
    \end{align}
    Let $P_1$, $P_2$ and $P_3$ denote the
    natural $\qb$-equivariant projections onto the subspaces
    of $\sym_\bc^2(\bh)$ spanned by $\set{e_1,e_2}$
    $\set{e_2,e_3}$ and $\set{e_1,e_3}$, respectively.
    A direct computation shows that
    $\ker P_i\cap\triangle(\bh)=\set{0}$
    for $i=1,2,3$. Consequently, $f_i=P_i\circ\triangle$,
    $i=1,2,3$, are admissible $\qb$-equivariant maps.
  \end{example}

  \begin{remark}\label{rmk:congruence}
    \Cref{example:q8} shows the existence
    of an admissible 2-homogeneous $\qb$-equivariant
    map $f_1:(\bh,\qb)\to(\bc^2,\qb/\bz_2)$,
    where $\bc^2\subset\sym_\bc^2(\bh)$
    is the subrepresentation spanned by
    $\set{e_1,e_2}$.
    In addition, $\alpha(\qb,\bh)=8$ since
    $\qb$ acts freely on $S(\bh)$.
    Therefore, it follows from
    the congruence principle that for any
    admissible $\qb$-equivariant map $f:(\bh,\qb)\to(\bc^2,\bh/\bz_2)$,
    \begin{align*}
      \deg(f)\equiv\deg(f_1)=2^2=4\pmod{8}.
    \end{align*}
    In particular, $\deg(f)$ is different from 0.
  \end{remark}

  In addition to the congruence principle,
  one can also analyze \Cref{example:q8}
  by the following result (see \cite{Atiyah-Tall}).

  \begin{thm}[Atiyah-Tall]\label{thm:atiyah-tall}
    Let $G$ be a finite $p$-group
    and $V$ and $W$ two $G$-representations.
    There exists an admissible equivariant
    map $f:V\to W$ with
    $\deg(f)\not\equiv0\pmod{p}$
    if and only if the irreducible
    components of $V$ and $W$ are
    Galois conjugate in pairs.
  \end{thm}

  \begin{remark}\label{rmk:atiyah-tall}
    Since $(\bh,\qb)$ is irreducible
    while $(\bc^2,\qb/\bz_2)$ is not,
    the irreducible components of $(\bh,\qb)$ and
    $(\bc^2,\qb/\bz_2)$ are not Galois conjugate in
    pairs. It follows from \Cref{thm:atiyah-tall} that
    \begin{align*}
      \deg(f)\equiv0\pmod{2}
    \end{align*}
    for any admissible $\qb$-equivariant map
    $f:(\bh,\qb)\to(\bc^2,\qb/\bz_2)$.
  \end{remark}

  \noindent
  A comparison between
  \Cref{rmk:congruence}
  and \Cref{rmk:atiyah-tall}
  shows that the result for \Cref{example:q8}
  obtained from the congruence principle
  is more informative.

  Possible extensions of \Cref{example:q8} to
  arbitrary $p$-groups were suggested
  by A. Kushkuley (see \cite{AED}).
  On the other hand,
  notice that $(\bh,\qb)$ (resp.~$(\bc^2,\qb/\bz_2)$)
  is induced by the one-dimensional
  representation $(\bc,\bz_4)$ (resp.~$(\bc,\bz_4/\bz_2)$).
  Furthermore, $\psi(z)=z^2$ is an
  admissible 2-homogeneous $Z_4$-equivariant map
  from $(\bc,\bz_4)$ to $(\bc,\bz_4/\bz_2)$
  and $f_1$ (see \Cref{example:q8}) is in fact the $\qb$-equivariant
  extension of $\psi$
  (see \Cref{fig:diagram}).

    \begin{figure}[h!]
    \centering
    \begin{tikzpicture}
      \node () at (-3.2,1.5) {$(\bh,\qb)$};
      \node () at (3.6,1.5) {$(\bc^2,\qb/\bz_2)$};
      \node () at (-3.2,0) {$(\bc,\bz_4)$};
      \node () at (3.6,0) {$(\bc,\bz_4/\bz_2)$};
      \node () at (-2.4,1.5) {$\owns$};
      \node () at (2.5,1.5) {$\in$};
      \node () at (-2.4,0) {$\owns$};
      \node () at (2.5,0) {$\in$};
      \node (H) at (-2,1.5) {$gz$};
      \node (C2) at (2,1.5) {$gz^2$};
      \node (C) at (-2,0) {$z$};
      \node (CC) at (2,0) {$z^2$};

      \draw[thick,->] (H)--(C2) node[pos=0.5, below] {$f_1$};
      \draw[thick,->] (C)--(CC) node[pos=0.5, below] {$\psi$};
      \draw[thick,->] (C)--(H) node[pos=0.5, left] {induced representation};
      \draw[thick,->] (CC)--(C2) node[pos=0.5, right] {induced representation};
    \end{tikzpicture}
    \caption{$f_1$ as an extension of $\psi$}
    \label{fig:diagram}
  \end{figure}

  The above example shows that if a $G$-representation $V$
  is induced from an $H$-representation $U$ ($H<G$),
  and $f$ is an $H$-equivariant homogeneous admissible map defined on $U$,
  then $f$ can be canonically extended to a $G$-equivariant homogeneous
  admissible map defined on $V$.
  However, the construction of an admissible homogeneous
  $G$-equivariant map becomes more involved
  if we are given a representation
  which is not induced from a subgroup.
  The example considered in the next subsection
  suggests a method to deal with this problem
  in several cases.
  

  \subsection{Example: $S_5$-representations}
  \label{subsection:s5example}
  In this subsection, we will construct
  an admissible 2-homogeneous equivariant
  map from $V^-$ to $V$ (see \Cref{rmk:s5})
  following \Crefrange{step:A}{step:C}, which
  will be illustrated in detail.
  In what follows, for an
  $S_5$-representation $X$ and $\sigma\in S_5$,
  denote by $\rho_X(\sigma)$
  and $\chi_X(\sigma)$ the corresponding
  matrix representation and character,
  respectively.

  \begin{enumerate}
  \renewcommand\labelenumi{(S\arabic{enumi})}
  \item Denote $U:=\sym^2(V^-)$.
    Recall that
    $\chi_{U}(\sigma)=\frac{1}{2}(\chi_{V^-}(\sigma)^2+\chi_{V^-}(\sigma^2))$
    (see, for example, \cite{Serre})
    and using \Cref{tbl:s5_chartbl},
    one has $U=\trep{S_5}\oplus V\oplus V_5$.

  \item One can take the linear
    equivariant map $A:U\to V$
    to be the orthogonal projection,
    which is also given by
    \begin{align}\label{eq:projection}
      A=\frac{\dim(V)}{\abs{S_5}}
          \sum_{\sigma\in S_5}
          {\overline{\chi_V(\sigma)}\,\rho_U(\sigma)}.
    \end{align}
    Take basis $\mathcal{B}_{V^-}:=\set{e_i}_{i=1}^4$
    of $V^-$ and $\mathcal{B}_U:=
    \set{e_i\otimes e_j}_{1\leq i\leq j\leq 4}$ of $U$.
    To obtain $\rho_U(\sigma)$ corresponding to
    $\mathcal{B}_U$, $\sigma\in S_5$,
    it suffices to let
    \begin{align*}
      \hskip -1em
      \rho_{V^-}((12)):={\tiny
      \begin{bmatrix}
        0 & -1 & 0 & 0 \\
        -1 & 0 & 0 & 0 \\
        0 & 0 & -1 & 0 \\
        0 & 0 & 0 & -1
      \end{bmatrix}},\;
      \rho_{V^-}((12345)):={\tiny
      \begin{bmatrix}
        0 & 0 & 0 & -1 \\
        1 & 0 & 0 & -1 \\
        0 & 1 & 0 & -1 \\
        0 & 0 & 1 & -1
      \end{bmatrix}}
    \end{align*}
    be the matrices corresponding to $\mathcal{B}_{V^-}$,
    and use formula
    \begin{align*}
      \rho_U(\sigma)(e_i\otimes e_j)
      =\rho_{V^-}(\sigma)(e_i)\otimes\rho_{V^-}(\sigma)(e_j).
    \end{align*}
    Substitution of $\rho_U(\sigma)$, $\sigma\in S_5$,
    in \eqref{eq:projection} yields
    \begin{align*}
      A=
      \tiny
      \def\arraystretch{2}
      \begin{bmatrix}
        \frac{3}{5} & -\frac{1}{5} & -\frac{1}{5} & -\frac{1}{5} &
            -\frac{2}{5} & -\frac{1}{5} & -\frac{1}{5} & -\frac{2}{5} &
            -\frac{1}{5} & -\frac{2}{5} \\
        -\frac{2}{15} & \frac{4}{15} & \frac{4}{15} & \frac{4}{15} &
            -\frac{2}{15} & \frac{4}{15} & \frac{4}{15} & \frac{8}{15} &
            \frac{4}{15} & \frac{8}{15} \\
        -\frac{2}{15} & \frac{4}{15} & \frac{4}{15} & \frac{4}{15} &
            \frac{8}{15} & \frac{4}{15} & \frac{4}{15} & -\frac{2}{15} &
            \frac{4}{15} & \frac{8}{15} \\
        -\frac{2}{15} & \frac{4}{15} & \frac{4}{15} & \frac{4}{15} &
            \frac{8}{15} & \frac{4}{15} & \frac{4}{15} & \frac{8}{15} &
            \frac{4}{15} & -\frac{2}{15} \\
        -\frac{2}{5} & -\frac{1}{5} & -\frac{1}{5} & -\frac{1}{5} &
            \frac{3}{5} & -\frac{1}{5} & -\frac{1}{5} & -\frac{2}{5} &
            -\frac{1}{5} & -\frac{2}{5} \\
        \frac{8}{15} & \frac{4}{15} & \frac{4}{15} & \frac{4}{15} &
            -\frac{2}{15} & \frac{4}{15} & \frac{4}{15} & -\frac{2}{15} &
            \frac{4}{15} & \frac{8}{15} \\
        \frac{8}{15} & \frac{4}{15} & \frac{4}{15} & \frac{4}{15} &
            -\frac{2}{15} & \frac{4}{15} & \frac{4}{15} & \frac{8}{15} &
            \frac{4}{15} & -\frac{2}{15} \\
        -\frac{2}{5} & -\frac{1}{5} & -\frac{1}{5} & -\frac{1}{5} &
            -\frac{2}{5} & -\frac{1}{5} & -\frac{1}{5} & \frac{3}{5} &
            -\frac{1}{5} & -\frac{2}{5} \\
        \frac{8}{15} & \frac{4}{15} & \frac{4}{15} & \frac{4}{15} &
            \frac{8}{15} & \frac{4}{15} & \frac{4}{15} & -\frac{2}{15} &
            \frac{4}{15} & -\frac{2}{15} \\
        -\frac{2}{5} & -\frac{1}{5} & -\frac{1}{5} & -\frac{1}{5} &
            -\frac{2}{5} & -\frac{1}{5} & -\frac{1}{5} & -\frac{2}{5} &
            -\frac{1}{5} & \frac{3}{5}
      \end{bmatrix}.
    \end{align*}

  \item The column vectors of the projection
    matrix $A$ span $V$ and one can
    obtain a basis $\mathcal{B}_V$
    of $V\subset U$ using the
    {\em Gram-Schmidt process}.
    With $\mathcal{B}_{V^-}$ and
    $\mathcal{B}_V$,
    $\phi:=A\circ\triangle$
    can be viewed as a map from
    $\bc^4$ to $\bc^4$.
    To be more explicit,
    $\phi=\left[\phi_1,\dots,\phi_{4}\right]^T$,
    where $\phi_i=\phi_i(x_1,x_2,x_3,x_4)$, $i=1,\dots,4$,
    is a 2-homogeneous polynomial.
    Denote $\mathcal{P}:=\set{\phi_i}_{i=1}^4$ and
    $\mathcal{P}_k:=\set{\phi_i|_{x_k=1}}_{i=1}^4$.
    Then, $\phi$ admits no non-trivial zero, i.e.,
    $\mathcal{P}$
    admits no non-trivial common zeros,
    if and only if $\mathcal{P}_k$
    admits no common zeros, i.e.,
    the Gr\"obner basis of the ideal
    generated by $\mathcal{P}_k$ contains
    only $1$, for $k=1,2,3,4$ (see \Cref{prop:null}).
    One can use Mathematica to show that
    it is indeed the case and
    thereby $\phi$ is admissible.
  \end{enumerate}

  \noindent
  To summarize, one has

  \begin{prop}\label{prop:s5}
    There exists an admissible 2-homogeneous
    $S_5$-equivariant map $\phi:V^-\to V$.
  \end{prop}

  The construction of an admissible
  2-homogeneous equivariant map $\phi$
  in \Cref{prop:s5},
  which involves computing the
  orthogonal projection matrix
  and verfying the criterion provided by
  \Cref{prop:admissible},
  is an ad hoc approach;
  it is difficult to obtain
  a global result for arbitrary $S_n$
  by this kind of constructions.
  In fact, for applications
  similar to \Cref{coro:sn},
  it suffices to know that
  an admissible homogeneous
  equivaraint map exists while its
  explicit formula is not necessary.
  In the next two subsections,
  we will employ a more universal
  technique to extend \Cref{prop:s5}
  to $S_n$ for arbitrary odd $n$.

  \subsection{Extension of \Cref{prop:s5}}
  Let us describe explicitly the
  setting to which we want to extend
  \Cref{prop:s5}.

  \begin{thm}\label{thm:sn}
    Let $(S_n;[n])$ be the natural action of
    the symmetric group $S_n$ on the set
    $[n]=\set{1,\dots,n}$
    and $V,V^-$ be the modules corresponding to the
    irreducible representations $\rho^a_{(S_n,[n])},
    \rho^a_{(S_n,[n])}\otimes 1_{S_n}^-$, repectively
    ($\sgnrep{S_n}$ is the sign representation).
    Assume that $n$ is odd.
    Then, there exists
    an admissible 2-homogeneous equivariant
    map from $V^-$ to $V$.
  \end{thm}

  \noindent
  The following statement is a starting
  point for proving \Cref{thm:sn}.

  \begin{prop}\label{prop:reduce}
    Let $V$, $V^-$ be as in \Cref{thm:sn}
    and $W$ an arbitrary $S_n$-representation.
    Then, there exists an admissible
    2-homogeneous equivariant map
    from $V^-$ to $W$ if and only
    if there exists an admissible
    2-homogeneous equivariant map
    from $V$ to $W$.
  \end{prop}

  \begin{proof}
    By taking the standard basis in
    $V$ (resp.~$V^-$), any map defined on
    $V$ (resp.~$V^-$) can be identified with
    a map on $\bc^{n-1}$.
    Let $\rho_V(\sigma)$ ($\sigma\in S_n$)
    be matrices representing the $S_n$-action
    on $V$.  Since $V^-=V\otimes\sgnrep{S_n}$,
    one can use the simple character argument to show that
    the formula
    \begin{align}
      \label{eq:Vminus_action}
      \rho_{V^-}(\sigma):=\begin{cases}
        \rho_V(\sigma),&\mbox{if }\sigma\in A_n,\\
        -\rho_V(\sigma),&\mbox{if }\sigma\in S_n\setminus A_n,
      \end{cases}
    \end{align}
    defines matrices representing the $S_n$-action
    on $V^-$.

    Assume that $\phi:V\to W$ is an admissible
    2-homogeneous equivariant map.  Then,
    \begin{align*}
      \phi(\rho_{V^-}(\sigma)v)=\begin{cases}
        \phi(\rho_V(\sigma)v),&\sigma\in A_n\\
        \phi(-\rho_V(\sigma)v),&\sigma\in S_n\setminus A_n
      \end{cases}
      =\phi(\rho_V(\sigma)v)=\rho_W(\sigma)\phi(v).
    \end{align*}
    Therefore, $\phi$ can be viewed as an admissible
    $2$-homogeneous equivariant map from $V^-$ to $W$
    as well.
    Similarly, one can show that
    if $\psi:V^-\to W$ is an admissible
    2-homogeneous equivariant map, then
    $\psi$ can be viewed as an admisslbe 2-homogeneous
    equivariant map from $V$ to $W$ as well.
    The result follows.
  \end{proof}

  \noindent
  \Cref{prop:reduce} reduces the
  proof of \Cref{thm:sn} to providing
  an admissible 2-homogeneous
  equivariant map from $V$ to $V$.
  Clearly, the later problem is equivalent
  to the existence of a bi-linear
  commutative (not necessarily associative)
  mulitplication
  $*:V\times V\to V$ commuting with
  the $G$-action on $V$ such that
  the algebra $(V,*)$ does not have
  2-nilpotents.
  In the next subsection,
  the existence
  of such multiplication will be
  studied using the Norton algebra
  techniques.
  
  \subsection{Norton Algebras without 2-nilpotents}\label{subsec:Norton}
  In this subsection, we will recall
  the construction of the Norton Algebra
  (see also \cite{Cameron-Goethals-Seidel})
  and apply related techniques to prove \Cref{thm:sn}.

  Let $\Omega$ be a finite $G$-set
  ($\abs{\Omega}=n$) and $U$ the
  associated permutation representation.
  With the standard basis
  $\set{e_g}_{g\in G}$,
  $u\in U$ can be viewed as
  a vector in $\bc^n$ and,
  hence, $U$ is endowed
  with the natural
  componentwise multiplication
  $u\cdot v:=[u_1v_1,\dots,u_nv_n]^T$
  and the scalar product
  $\left<u,v\right>:=\sum_{i=1}^n{u_i\overline{v_i}}$
  for $u=[u_1,\dots,u_n]^T$ and $v=[v_1,\dots,v_n]^T$.
  Then, $(U,\cdot)$ is a commutative and
  associative algebra with the $G$-action
  commuting with the multiplication ``$\cdot$".

  Let $W\subset U$ be a non-trivial $G$-invariant
  subspace.  Denote by $P:U\to W$ the orthogonal
  projection with respect to $\left<\cdot,\cdot\right>$
  and define the Norton algebra $(W,\star)$
  as follows: $w_1\star w_2:=P(w_1\cdot w_2)$
  for any $w_1,w_2\in W$.
  It is clear that the Norton algebra is
  commutative but not necessarily associative
  complex algebra with the $G$-action commuting
  with the multiplication $\star$.
  In particular, the quadratic map $w\mapsto w\star w$
  is $G$-equivariant on $W$.

  In connection to \Cref{thm:sn}, consider
  $G=S_n$.  Recall that $S_n$ acts naturally on
  $\bn_n:=\set{1,\dots n}$ by permutation.
  Let $\Omega$ be the set of all two-element
  subsets of $\bn_n$, i.e.,
  $\Omega:=\set{\set{i,j}:1\leq i<j\leq n}$,
  on which $S_n$ acts by
  $\sigma(\set{i,j})=\set{\sigma(i),\sigma(j)}$
  for any $\sigma\in S_n$ and $\set{i,j}\in\Omega$.
  In this case, the permutation representation
  associated to $\Omega$ is $U=\trep{S_n}\oplus W\oplus W'$,
  where $\trep{S_n}$, $W$ and $W'$ are
  irreducible $S_n$-representations
  with $\dim(\trep{S_n})=1$,
  $\dim(W)=n-1$ and $\dim(W')=n(n-3)/2$.

  Denote by $\mathcal{B}:=\set{e_{ij}:1\leq i<j\leq n}$
  the standard basis of $U$ and
  let $f_i:=\sum_{j\neq i}{e_{ij}}$.
  Since $\set{f_i:i\in\bn_n}$ is linearly independent
  and $\sigma(f_i)=f_{\sigma(i)}$
  for any $\sigma\in S_n$ and $i\in\bn_n$,
  $\set{f_i:i\in\bn_n}$ forms a basis
  of an $S_n$-subrepresentation $F$.
  To be more explicit, $F=\trep{S_n}\oplus W$,
  where $\trep{S_n}=\set{z\sum_{i=1}^n{f_i}:z\in\bc}$
  and $W=\set{\sum_{i=1}^n{z_if_i}:z_i\in\bc,
  \sum_i{z_i}=0}$.
  Note that $\set{f_n-f_j:j\in\bn_{n-1}}$ forms
  a basis of $W$.
  Then, $(W,\star)$ is a Norton algebra
  satisfying the following property.

  \begin{prop}\label{prop:norton}
    The Norton algebra $(W,\star)$ admits
    $2$-nilpotents if and only if $n$ is even.
    In such a case, $w\star w=0$ if and only
    if $w=\alpha(\sum_{i\in I}{f_i}-\sum_{i\notin I}{f_i})$
    for some $\alpha\in\bc$ and $I\subset\bn_n$
    with $\abs{I}=n/2$.
  \end{prop}

  \begin{proof}
    Let $P:U\to W$ be the orthogonal projection.
    We have to solve the equation $P(w^2)=0$
    for $w\in W$.  Note that $P(w^2)=0$
    if and only if $\left<w^2,f_n-f_i\right>=0$
    for any $i\in\bn_{n-1}$
    or, equivalently,
    \begin{align}\label{eq:cond}
      \left<w^2,f_i\right>=c
    \end{align}
    for any $i\in\bn_n$ and some $c\in\bc$.
    Set $w=\sum_{i=1}^n{z_if_i}$ with
    $\sum_{i=1}^n{z_i}=0$.
    Then,
    \begin{align*}
      w=\;&\sum_{1\leq i<j\leq n}{(z_i+z_j)e_{ij}},\\
      w^2=\;&\sum_{1\leq i<j\leq n}{(z_i+z_j)^2e_{ij}}.
    \end{align*}
    On the other hand, since
    \begin{align*}
      \left<e_{ij},f_k\right>=\begin{cases}
        1,&k\in\set{i,j},\\
        0,&k\notin\set{i,j},
      \end{cases}
    \end{align*}
    $\left<w^2,f_k\right>=\sum_{i\neq k}{(z_i+z_k)^2}$
    and it follows from \eqref{eq:cond} that
    \begin{align*}
      c=\;&\sum_{i\neq k}(z_i+z_k)^2=\sum_{i=1}^n(z_i+z_k)^2
      -4{z_k}^2=\sum_{i=1}^n{z_i}^2+2z_k\sum_{i=1}^nz_i+
      (n-4){z_k}^2\\
      =\;&\sigma_2+(n-4){z_k}^2,
    \end{align*}
    where $k\in\bn_n$ and $\sigma_2:=\sum_{i=1}^n{z_i}^2$.
    Hence, ${z_k}^2$ is independent of $k$ and
    ${z_k}^2=\sigma_2/n$.
    Let $\alpha$ be an arbitrary (complex)
    root of the equation $z^2=\sigma_2/n$.
    Then, $z_k=\pm\alpha$. Denote by
    $I\subset\bn_n$ the set of indices $k$
    such that $z_k=\alpha$.
    Since $\sum_{i=1}^n{z_i}=0$,
    one has either

    \begin{itemize}
    \item $\alpha=0$, or
    \item $\alpha\neq0$, $n$ is even and $\abs{I}=n/2$.
    \end{itemize}

    \noindent
    The result follows.
  \end{proof}

  \begin{proof}[Proof of \Cref{thm:sn}]
    The result simply follows from \Cref{prop:reduce}
    and \Cref{prop:norton}.
  \end{proof}

  \section{Applications to Congruence Principle}\label{sec:7}

  \subsection{The Brouwer Degree}
  Recall the construction of the Brouwer degree.
  Let $M$ and $N$ be compact, connected,
  oriented $n$-dimensional manifolds (without boundary),
  and let $f:M\to N$ be a smooth map.
  Let $y\in N$ be a regular value of $f$.
  Then, $f^{-1}(y)$ is either empty
  (in which case, define the Brouwer degree of $f$ to be zero),
  or consists of finitely many points, say $x_1,\dots,x_k$.
  In the latter case, for each $i=1,\dots,k$,
  take the tangent spaces $T_{x_i}M$ and $T_{y}N$
  with the corresponding orientations.
  Then, the derivative $D_{x_i}f:T_{x_i}M\to T_{y}N$
  is an isomorphism.
  Define the Brouwer degree by the formula
  \begin{equation}\label{eq:Brouwer-degree}
    \deg(f)=\deg(f,y):=\sum_{i=1}^k\mathrm{sign}(\det(D_{x_i}f)).
  \end{equation}
  It is possible to show that $\deg(f,y)$ is independent of
  the choice of a regular value $y\in N$
  (see, for example, \cite{Nirenberg,Deimling,Milnor}). 
  If $f:M\to N$ is continuous,
  then one can approximate $f$ by a smooth map
  $g:M\to N$ and take $\deg(g)$ to be the Brouwer degree
  of $f$ (denoted $\deg(f)$).
  Again, $\deg(f)$ is independent of a close approximation.

  Finally, let $M$ be as above and let $W$ be the
  oriented Euclidean space such that $\dim M=\dim W-1$.
  Given a continuos map $f:M\to W\setminus\set{0}$,
  define the map $\widetilde{f}:M\to S(W)$ by
  $\widetilde{f}(x)={f(x)\over\norm{f(x)}}$ ($x\in M$).
  Then, $\deg(\widetilde{f})$ is correctly defined.
  Set $\deg(f):=\deg(\widetilde{f})$ and call it the Brouwer degree of $f$.
  In particular, if $V$ and $W$ are oriented
  Euclidean spaces of the same dimension and $f:V\to W$ is admissible,
  then $f$ takes $S(V)$ to $W\setminus\set{0}$.
  Define the Brouwer degree of $f$ by $\deg(f):=\deg(f|_{S(V)})$. 
      
  \subsection{Congruence Principle for Solvable Groups}
  Combining \Cref{thm:solvable_groups} and
  the congruence principle, one immediately obtains
  the following result.

  \begin{coro}\label{coro:maps-represent-solv}
  Let $G$ be a solvable group and let $V$ and $W$ be two
  $n$-dimensional representations. Assume, in addition, that $V$ is 
  non-trivial and irreducible, and suppose that there exists an equivariant map
  $\Phi:S(V)\to W\setminus\set{0}$.
  Then, $\alpha(V)>1$ and for any  equivariant map
  $\Psi:S(V)\to W\setminus\set{0}$, one has  
  \begin{align}
    \deg(\Psi)\equiv\deg(\Phi)\pmod{\alpha(V)}
  \end{align}
  \end{coro}

  \noindent
  In addition, one has the following 

  \begin{coro}
  Let $G$ be a solvable group and $V,W\in\irrnt(G)$.
  If $V$ and $W$ are Galois-equivalent, then
  $\alpha(V)>1$ and
  $\deg(f)\not\equiv0\pmod{\alpha(V)}$
  for any $G$-equivariant map $f:S(V)\to S(W)$.
  \end{coro}

  \begin{proof}
  Take a $G$-equivariant map $f:S(V)\to S(W)$.
  Since $V$ and $W$ are Galois-equivalent,
  one has $\dim V^H=\dim W^H$ for any $H<G$.
  So, there exists a $G$-equivariant map
  $g:S(W)\to S(V)$
  (see, for example, \cite{Bartsch,tomDieck,Kushkuley-Balanov}).
  In this case, $g\circ f:S(V)\to S(V)$ is a $G$-equivariant map
  and, by the congruence principle,
  $\deg(g\circ f)\equiv 1\pmod{\alpha(V)}$.
  Since $\alpha(V)>1$ (see \Cref{thm:solvable_groups}),
  the result follows.
  \end{proof}
  
  \begin{coro}\label{coro:maps-manifold-solv}
  Let $G$ be a solvable group,
  $W$ an $n$-dimensional irreducible
  (complex) $G$-representation and
  $M$ a (real) compact, connected, oriented
  smooth $2n-1$-dimensional $G$-manifold. 
  Assume, in addition,  that 
  \begin{equation}\label{eq:existence}
    \dim_\br M^H\leq\dim_\br W^H-1\mbox{ for any }(H)\in\Phi(G,M).
  \end{equation}
  Then: 
  \begin{enumerate}
  \item[(i)] there exists an equivariant map $f:M\to W\setminus\set{0}$;
  \item[(ii)] $\alpha(M):=\gcd\set{\abs{G(x)}:x\in M}>1$;
  \item[(iii)] for any equivariant map $g:M\to W\setminus\set{0}$, one has
    \begin{align}\label{eq:congruence}
      \deg(g)\equiv\deg(f)\pmod{\alpha(M)}.
    \end{align}
  \end{enumerate}
  \end{coro}

  \begin{proof}
  By condition \eqref{eq:existence},
  there exists an equivariant map  $f:M\to W\setminus\set{0}$
  (see, for example, \cite{tomDieck,Kushkuley-Balanov}). 
  Therefore, $G_{f(x)}\geq G_x$ so that $\abs{G(f(x))}$ divides
  $\abs{G(x)}$ for any $x\in M$.
  Since $\alpha(W)>1$ (see \Cref{thm:solvable_groups}),
  it follows that $\alpha(M)>1$.
  Finally, the congruence principle shows that
  $\eqref{eq:congruence}$ is true.
  \end{proof}


  \begin{remark}\label{rem:possible-extensions}
  By combining \Cref{thm:solvable_groups} with other versions
  of the congruence principle given in \cite{Kushkuley-Balanov},
  one can easily obtain many other results on degrees of
  equivariant maps for solvable groups.
  We leave this task to a reader.
  \end{remark}
  \subsection{Congruence Principle for $S_n$-representations}
  In this subsection, we will study the Brower degree
  of equivariant maps from $S(V^-)$ to $S(V)$,
  where $V$ and $V^-$ are as in \Cref{thm:sn}.
  To this end, we need the following proposition.

  \begin{prop}\label{prop:alpha_V}
  Let $V$ and $V^-$ be $S_n$-modules as in \Cref{thm:sn}
  with $n=p^k>3$ being an odd prime power.  Then:

  \begin{enumerate}
  \renewcommand{\labelenumi}{(\roman{enumi})}
  \item $\alpha(V)=p$;
  \item $\alpha(V^-)=2p$.
  \end{enumerate}
  \end{prop}

  \begin{proof}
  (i) According to \Cref{lemma:2trans_groups},
  one has that $p$ divides $\alpha(V)$.
  Hence, it suffices to show that $V\setminus\set{0}$ admits two $S_n$-orbits,
  say $O_1$ and $O_2$, such that $\gcd\set{\abs{O_1},\abs{O_2}}=p$.
  Indeed, let $O_1=S_n(x)$ and $O_2=S_n(y)$ be two $S_n$-orbits
  in $V\setminus\set{0}$ with
  \begin{align}
    \label{eq:x}
    x=\;&(n-1,\underbrace{-1,\dots,-1}_{n-1}),\\
    \label{eq:y}
    y=\;&(\underbrace{p-1,\dots,p-1}_{p^k},
        \underbrace{-1,\dots,-1}_{(p-1)p^{k-1}}).
  \end{align}
  Then, $\abs{O_1}=p^k$, $\abs{O_2}={p^k\choose p^{k-1}}$,
  from which it follows that $\gcd\set{\abs{O_1},\abs{O_2}}=p$.\\

  \noindent
  (ii)
  Since there exists an admissible
  equivariant map from $V^-$ to $V$ (see \Cref{thm:sn}),
  one has that $p$ divides $\alpha(V^-)$.
  Hence, it suffices to show that

  \begin{enumerate}
  \renewcommand{\labelenumi}{(\alph{enumi})}
  \item any $S_n$-orbit in $V^-\setminus\set{0}$ is of even length;
  \item $V^-\setminus\set{0}$ admits two $S_n$-orbits, say $O_1$ and $O_2$,
    such that $\gcd\set{\abs{O_1},\abs{O_2}}=2p$.
  \end{enumerate}

  For (a), take an $S_n$-orbit $O\subset V^-\setminus\set{0}$.
  If the transposition $(12)$ acts on $O$ without fixed points,
  then $\abs{O}$ is even.
  Otherwise, let $x$ be a vector fixed by $(12)$.
  Then, $x=(a,-a,0,\dots,0)$ for some $a\neq0$ (see \eqref{eq:Vminus_action}).
  Since $n>3$, one has $-x\in O$, from which it follows that $-O=O$.
  Thus, the involution $x\mapsto-x$ acting on $O$ is
  without fixed points, which implies that $\abs{O}$ is even
  (note that this argument does not work when $n=3$,
  in which case $\abs{O}=3$).

  For (b), let $O_1=S_n(x)$ and $O_2=S_n(y)$ be two $S_n$-orbits in
  $V^-\setminus\set{0}$ with $x$ and $y$ given in
  \eqref{eq:x} and \eqref{eq:y}, respectively.
  Observe that $-x\in O_1$ and $-y\in O_2$ (by transposing two
  $-1$ components), from which it follows that $\abs{O_1}=2p$ and
  $\abs{O_2}=2{p^k\choose p^{k-1}}$.
  Therefore, $\gcd\set{\abs{O_1},\abs{O_2}}=2p$.
  \end{proof}

  Combining the congruence principle and \Cref{prop:alpha_V} yields:
  
  \begin{coro}\label{coro:sn}
  Suppose that $n=p^k>3$ is an odd prime power.
  For any $S_n$-equivariant map $\Psi:S(V^-)\to V\setminus\set{0}$,
  \begin{align*}
    \deg(\Psi)\equiv2^{n-1}\pmod{2p}
  \end{align*}
  (in particular, $\deg(\Psi)\neq0$).
  \end{coro}

  \begin{proof}
  Let $\Phi:S(V^-)\to V\setminus\set{0}$ be the
  $2$-homogeneous equivariant map provided by \Cref{thm:sn}.
  Since $\alpha(S_n,V^-)=2p$ (see \Cref{prop:alpha_V})
  and $\deg(\Phi)=2^{n-1}$, it follows from the congruence principle that
  \begin{align*}
    \deg(\Psi)\equiv\deg(\Phi)=2^{n-1}\pmod{2p}.
  \end{align*}
  \end{proof}


%
%
    
  \clearpage
  \appendix
  \section{Tables}
    \begin{table}[h!]
    \centering
    \begin{tabularx}{\textwidth}{@{}LCCCYYYYY@{}}
      \toprule
      \multirow{2}{*}{$\rho$} & \multirow{2}{*}{\quad$\alpha(\rho)$\quad} &
          \multicolumn{2}{c}{(2-transitivity)} &
          \multicolumn{5}{c}{\multirow{2}{*}{character}} \\
      & & H\quad & \subgind{G}{H}\quad & & & & & \\
      \midrule
      \psi_0 & 1 & & & 1 & 1 & 1 & 1 & 1 \\
      \psi_1 & 2 & & & 3 & -1 & . &
          \frac{1-\sqrt{5}}{2} &
          \frac{1+\sqrt{5}}{2} \\
      \psi_2 & 2 & & & 3 & -1 & . &
          \frac{1+\sqrt{5}}{2} &
          \frac{1-\sqrt{5}}{2} \\
      \psi_3 & 5 & A_4 & 5 & 4 & . & 1 & -1 & -1 \\
      \psi_4 & 1 & D_5 & 6 & 5 & 1 & -1 & . & . \\
      \bottomrule
    \end{tabularx}
    \caption{character table of $G=A_5$}
    \label{tbl:a5_chartbl}
  \end{table}

    \begin{table}[h!]
    \centering
    \begin{tabularx}{\textwidth}{@{}LCCCYYYYYYY@{}}
      \toprule
      \multirow{2}{*}{$\rho$} & \multirow{2}{*}{\quad$\alpha(\rho)$\quad} &
          \multicolumn{2}{c}{(2-transitivity)} &
          \multicolumn{7}{c}{\multirow{2}{*}{character}} \\
      & & H\quad & \subgind{G}{H}\quad & & & & & & & \\
      \midrule
      \xi_1 & 2 & A_5 & 2 &
          1 & -1 &  1 &  1 & -1 & -1 &  1 \\
      \xi_2 & 10 & & &
          4 & -2 &  . &  1 &  1 &  . & -1 \\
      \xi_3 & 1 & \aff(\bz_5) & 6 &
          5 & -1 &  1 & -1 & -1 &  1 &  . \\
      \xi_4 & 2 & & &
          6 &  . & -2 &  . &  . &  . &  1 \\
      \xi_5 & 1 & & &
          5 &  1 &  1 & -1 &  1 & -1 &  . \\
      \xi_6 & 5 & S_4 & 5 &
          4 &  2 &  . &  1 & -1 &  . & -1 \\
      \xi_0 & 1 & & &
          1 &  1 &  1 &  1 &  1 &  1 &  1 \\
      \bottomrule
    \end{tabularx}
    \caption{character table of $G=S_5$}
    \label{tbl:s5_chartbl}
  \end{table}

    \begin{table}[h!]
    \centering
    \begin{tabularx}{\textwidth}{@{}CCCYYYYYYYYYYY@{}}
      \toprule
      \rho_i=\rho_{(G;X_i)}^a & \alpha(\rho_i) & \abs{X_i} & \multicolumn{11}{c}{character} \\
      \midrule
      \rho_1 & 7 & 7 & 6 & 6 & 2 & 2 & 2 & . & . & . & . & -1 & -1 \\
      \rho_2 & 2 & 2^3 & 7 & -1 & 3 & -1 & -1 & 1 & -1 & 1 & -1 & . & . \\
      \rho_3 & 2 & 2^3 & 7 & 7 & -1 & -1 & -1 & -1 & -1 & 1 & 1 & . & . \\
      \rho_4 & 2 & 2^3 & 7 & -1 & -1 & 3 & -1 & -1 & 1 & 1 & -1 & . & . \\
      \bottomrule
    \end{tabularx}
    \caption{irreducible $(V\rtimes\gl(3,2))$-representations
        associated to 2-transitive actions}
    \label{tbl:gl_chartbl}
  \end{table}

    \begin{table}[h!]
    \centering
    \begin{tabularx}{\textwidth}{@{}LCYYYYYYYYY@{}}
      \toprule
      \rho & \alpha(\rho)\quad &
          \multicolumn{9}{c}{character} \\
      \midrule
      \psi_0 & 1 & 1 &  1 &  1 &  1 &  1 &  1 & 1 & 1 & 1 \\
      \psi_1 & 2 & 7 & -1 & -2 &  1 &  1 &  1 & . & . & . \\
      \psi_2 & 2 & 7 & -1 &  1 &  A &  C &  B & . & . & . \\
      \psi_3 & 2 & 7 & -1 &  1 &  B &  A &  C & . & . & . \\
      \psi_4 & 2 & 7 & -1 &  1 &  C &  B &  A & . & . & . \\
      \psi_5 & 3 & 8 &  . & -1 & -1 & -1 & -1 & 1 & 1 & 1 \\
      \psi_6 & 1 & 9 &  1 &  . &  . &  . &  . & D & F & E \\
      \psi_7 & 1 & 9 &  1 &  . &  . &  . &  . & E & D & F \\
      \psi_8 & 1 & 9 &  1 &  . &  . &  . &  . & F & E & D \\
      \bottomrule
    \end{tabularx}
    \caption{character table of $\mathrm{PSL}(2,8)$}
    \label{tbl:psl28_chartbl}
  \end{table}

    \begin{table}[h!]
    \centering
    \begin{tabularx}{\textwidth}{@{}LCYYYYYYYYYYY@{}}
      \toprule
      \rho & \alpha(\rho)\quad &
          \multicolumn{11}{c}{character} \\
      \midrule
      \xi_0    & 1 &  1 &  1 &  1 &  1 &   1 &   1 &    1 &    1 &  1 &    1 &    1 \\
      \xi_1    & 3 &  1 &  1 &  1 &  1 &   A & A^* &  A^* &    A &  1 &    A &  A^* \\
      \xi_2    & 3 &  1 &  1 &  1 &  1 & A^* &   A &    A &  A^* &  1 &  A^* &    A \\
      \xi_3    & 2 &  7 & -2 &  1 & -1 &   1 &   1 &   -1 &   -1 &  . &    1 &    1 \\
      \xi_4    & 6 &  7 & -2 &  1 & -1 & A^* &   A &   -A & -A^* &  . &  A^* &    A \\
      \xi_5    & 6 &  7 & -2 &  1 & -1 &   A & A^* & -A^* &   -A &  . &    A &  A^* \\
      \xi_6    & 3 &  8 & -1 & -1 &  . &   2 &   2 &    . &    . &  1 &   -1 &   -1 \\
      \xi_7    & 3 &  8 & -1 & -1 &  . &   B & B^* &    . &    . &  1 &   -A & -A^* \\
      \xi_8    & 3 &  8 & -1 & -1 &  . & B^* &   B &    . &    . &  1 & -A^* &   -A \\
      \xi_9    & 2 & 21 &  3 &  . & -3 &   . &   . &    . &    . &  . &    . &    . \\
      \xi_{10} & 1 & 27 &  . &  . &  3 &   . &   . &    . &    . & -1 &    . &    . \\
      \bottomrule
    \end{tabularx}
    \caption{character table of $\mathrm{P\Gamma L}(2,8)$}
    \label{tbl:pgammal28_chartbl}
  \end{table}

  \clearpage

  \section{Figures}\
    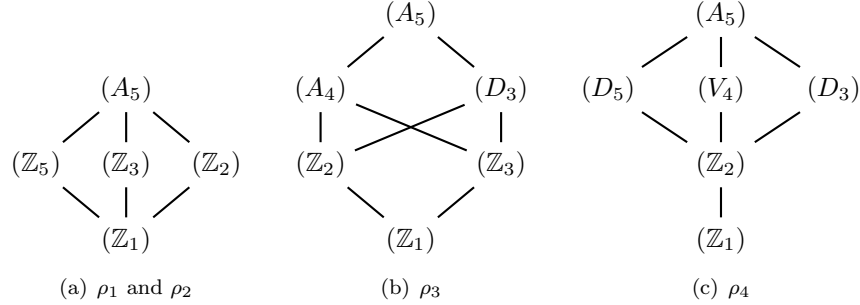
\begin{figure}[h!]
    \centering
    \subfigure[$\rho_1$ and $\rho_2$]{\begin{tikzpicture}
      \node (A5) at (0,2) {$(A_5)$};
      \node (Z5) at (-1.2,1) {$(\bz_5)$};
      \node (Z3) at (0,1) {$(\bz_3)$};
      \node (Z2) at (1.2,1) {$(\bz_2)$};
      \node (Z1) at (0,0) {$(\bz_1)$};

      \draw[thick] (A5)--(Z5);
      \draw[thick] (A5)--(Z3);
      \draw[thick] (A5)--(Z2);
      \draw[thick] (Z5)--(Z1);
      \draw[thick] (Z3)--(Z1);
      \draw[thick] (Z2)--(Z1);
    \end{tikzpicture}}
    \quad
    \subfigure[$\rho_3$]{\begin{tikzpicture}
      \node (A5) at (0,3) {$(A_5)$};
      \node (A4) at (-1.2,2) {$(A_4)$};
      \node (D3) at (1.2,2) {$(D_3)$};
      \node (Z2) at (-1.2,1) {$(\bz_2)$};
      \node (Z3) at (1.2,1) {$(\bz_3)$};
      \node (Z1) at (0,0) {$(\bz_1)$};

      \draw[thick] (A5)--(A4);
      \draw[thick] (A5)--(D3);
      \draw[thick] (A4)--(Z2);
      \draw[thick] (A4)--(Z3);
      \draw[thick] (D3)--(Z2);
      \draw[thick] (D3)--(Z3);
      \draw[thick] (Z2)--(Z1);
      \draw[thick] (Z3)--(Z1);
    \end{tikzpicture}}
    \quad
    \subfigure[$\rho_4$]{\begin{tikzpicture}
      \node (A5) at (0,3) {$(A_5)$};
      \node (D5) at (-1.5,2) {$(D_5)$};
      \node (D3) at (1.5,2) {$(D_3)$};
      \node (V4) at (0,2) {$(V_4)$};
      \node (Z2) at (0,1) {$(\bz_2)$};
      \node (Z1) at (0,0) {$(\bz_1)$};

      \draw[thick] (A5)--(D5);
      \draw[thick] (A5)--(D3);
      \draw[thick] (A5)--(V4);
      \draw[thick] (D5)--(Z2);
      \draw[thick] (D3)--(Z2);
      \draw[thick] (V4)--(Z2);
      \draw[thick] (Z2)--(Z1);
    \end{tikzpicture}}
    \caption{lattices of orbit types of $\rho\in\irrnt(A_5)$}
    \label{fig:a5_irr_lat_orbtypes}
  \end{figure}




  \bibliographystyle{abbrv}
  \bibliography{congruence_principle}

  \end{document}